\documentclass[12pt,a4paper]{article}
\usepackage[utf8]{inputenc}
\usepackage{amsmath}
\usepackage{amsfonts}
\usepackage{amssymb}
\usepackage{amsthm}
\usepackage{graphicx}
\usepackage{fancyhdr}
\usepackage{mathrsfs}
\usepackage{dsfont}
\usepackage{tikz}
\usepackage{float}
\usepackage{enumitem}
\usetikzlibrary{quotes,angles}
\newcommand*{\affaddr}[1]{#1} 
\newcommand*{\affmark}[1][1]{\textsuperscript{#1}}
\author{Julie Clutterbuck\affmark[a]\thanks{email: Julie.Clutterbuck@monash.edu} and James Larsen-Scott\affmark[a]\thanks{Corresponding author email: james.larsen-scott@monash.edu}\\
	 \affaddr{\small\affmark[a]School of Mathematics, Monash University, 9 Rainforest Walk,
	 	Melbourne, VIC 3800,
	 	Australia}}
\newcommand*{\bchi}{\mbox{\Large$\chi$}}
\newtheorem{thm}{Theorem}
\newtheorem{ques}{Question}
\newtheorem{lem}{Lemma}
\newtheorem{prop}{Proposition}
\newtheorem{cry}{Corollary}
\newtheorem*{remark}{Remark}
\date{September 6, 2023}
\begin{document}
	\title{Local Spectral Optimisation for Robin Problems with Negative Boundary Parameter on Quadrilaterals}
	\maketitle
	\begin{abstract}
		We investigate the Robin eigenvalue problem for the Laplacian with negative boundary parameter on quadrilateral domains of fixed area. In this paper, we prove that the square is a local maximiser of the first eigenvalue with respect to the Hausdorff metric. We also provide asymptotic results relating to the optimality of the square for extreme values of the Robin parameter.
	\end{abstract}
	{\let\thefootnote\relax\footnote{The work of James Larsen-Scott was supported via an Australian Government Research Training Program Scholarship. The research of Julie Clutterbuck was supported by grant  
			DP220100067 of the Australian Research Council.}}
	\section{Introduction}
	Consider the Robin eigenvalue problem for the Laplacian
	\begin{subequations}
		\begin{align}
			-\Delta u=\lambda u \quad \text{in } \Omega, \label{Eq:robA}\\
			\frac{\partial u}{\partial \nu}+\alpha u=0 \quad \text{on }\partial\Omega, \label{Eq:robB}
		\end{align}
	\end{subequations}
	with outward facing unit normal $\nu,$ eigenvalue $\lambda$ and Robin parameter $\alpha,$ for a domain $\Omega$ of $\mathbb{R}^N,N\geq 2.$ The Robin problem can be seen as an extension of the well-studied Neumann and Dirichlet boundary conditions, where $\alpha=0$ corresponds to the Neumann problem, and $\alpha=+\infty$ corresponds to the Dirichlet problem. The eigenvalues associated with the problem (\ref{Eq:robA}-\ref{Eq:robB}) form a non-decreasing sequence ${\lambda_k(\Omega,\alpha)}_{k=1}^{\infty}.$ The first eigenvalue $\lambda_1$ is simple and given by the variational characterisation
	\begin{align}
		\lambda_1(\Omega,\alpha)=\min_{\substack {u\in H^1(\Omega)\\ u\neq 0}}\frac{\int_{\Omega}|\nabla u|^2 dx+\alpha\int_{\partial\Omega}|u|^2d\sigma}{\int_{\Omega}|u|^2 dx},\label{eq:eig1}
	\end{align}
	where we understand the boundary integral in the sense of traces where necessary.\\
	This paper concerns optimisation problems for the first eigenvalue of the Laplacian. In particular we are interested in the following question
	\begin{ques}
		Let $\Lambda$ be a collection of domains with fixed area. For each fixed $\alpha\in\mathbb{R}$ does there exist $\Omega_0\in\Lambda$ such that for each $\Omega\in\Lambda$ we have $|\lambda_1(\Omega,\alpha)|\geq |\lambda_1(\Omega_0,\alpha)|$?
	\end{ques}
	A famous result of this type is the Faber-Krahn inequality: among domains of fixed dimension and area, the first eigenvalue of (\ref{Eq:robA}) with Dirichlet boundary condition $u=0$ on $\partial\Omega$ is minimized by the ball. This was initially conjectured by Lord Rayleigh \cite{RayleighJohnWilliamStruttBaron1842-19191878Ttos} before later being proved independently by Faber  \cite{faber1923beweis} and Krahn \cite{krahn1925rayleigh}.\\
	An analogous result holds for the Robin problem (\ref{Eq:robA}-\ref{Eq:robB}) with positive Robin parameter $\alpha>0$.
	This was established in two dimensions by Bossel \cite{bossel1986membranes}, before Daners later extended this work to higher dimensions \cite{DanersDaniel2006AFif} and alongside Bucur extended to the p-Laplacian \cite{BucurDorin2009Aaat}. For recent related results see \cite{alvino2023talenti,BucurDorin2018TqFi}.\\
	In the case of negative Robin parameter $\alpha<0,$ the first eigenvalue is negative $(\lambda_1<0)$ and the relevant spectral optimisation question is whether there is a domain $\Omega_0$ which maximises the first eigenvalue. Bareket conjectured (at least in two dimensions, see \cite{bareket1977isoperimetric}) that a reverse Faber-Krahn type inequality holds, where the ball maximises the first eigenvalue among domains of a given volume. However, this was shown to be false in general by Freitas and Krej\v ci\v r\'ik \cite{freitas2015first} who showed for sufficiently large negative $\alpha$ that a spherical shell has larger first eigenvalue. In the same paper, Freitas and Krej\v ci\v r\'ik also showed that in 2 dimensions a reverse Faber-Krahn equality does hold, where the ball is still the maximiser, provided that the absolute value of $\alpha$ is sufficiently small. It remains an open question to consider whether the disk is a maximizer once one restricts to simply connected domains in two dimensions. \\
	\begin{ques}
		Let $\Lambda$ be the collection of $N$-gons with fixed area. Does there exist $\Omega_0\in\Lambda$ such that for each $\alpha\in\mathbb{R}$ and $\Omega\in\Lambda$ we have $|\lambda_1(\Omega,\alpha)|\geq |\lambda_1(\Omega_0,\alpha)|$? If so, is $\Omega_0$ the regular $N$-gon?
	\end{ques}
	In the case of the Dirichlet problem, for triangles and quadrilaterals, P\'olya \cite{PolyaG1951Iiim} proved that indeed the equilateral triangle and square respectively minimize the first eigenvalue. However, for $N\geq 5$ the problem is still open in the Dirichlet case. See \cite{bogosel2022polygonal,indrei2022first} for recent developments. In the case of rectangles, for both positive and negative $\alpha$ the square is the optimiser, as can be seen through a separation of variables argument (see for instance the survey paper \cite{laugesen2019robin}.) However, for general quadrilaterals and triangles this problem is still open for both positive and negative $\alpha$ (this problem for triangles with $\alpha>0$ is conjecture 6.4 in \cite{henrot2017shape}).\\
	Recent work by Krej\v ci\v r\'ik et al. in \cite{krejvcivrik2023reverse} has made some progress on this problem for trianglar domains when the Robin parameter $\alpha$ is negative. The first result of this paper showed that for sufficently small $\alpha$ the equilateral triangle is a local maximiser. The second result showed, that for a fixed triangle of given area, if $\alpha$ is sufficiently small or large in absolute value, the  equilateral triangle of the same area has larger first eigenvalue. In this paper we follow a similar approach to Krej\v ci\v r\'ik et al. and obtain similar results for the Robin problem with negative boundary parameter on quadrilateral domains.\\
	Our main result is the following local result.
	\begin{thm}\label{thm1}
		The square is a local maximiser (with respect to the Hausdorff metric up to isometry) of the first eigenvalue of the problem (\ref{Eq:robA}-\ref{Eq:robB}) with negative boundary parameter $\alpha$ in the family of quadrilaterals of fixed area.
	\end{thm}
	Notice that unlike the equivalent result for triangles in \cite{krejvcivrik2023reverse} there is no restriction on the size of $\alpha.$ The method of proof follows that in \cite{krejvcivrik2023reverse}: we parametrise an arbitrary quadrilateral of fixed area $\Omega$ in terms of 4 parameters. We replace the problem (\ref{Eq:robA}-\ref{Eq:robB}) on $\Omega$ with an isospectral problem on the (rotated) square $\Omega_0$ whose associated quadratic form depends analytically upon the 4 parameters. By computing the first and second order derivatives with respect to the parameters of the eigenvalue, it can be shown that the square is a local maximiser in this parameter space. The argument takes advantage of the fact that the solution to (\ref{Eq:robA}-\ref{Eq:robB}) is known explicitly on the square. \\
	Our second result shows that a reverse Faber-Krahn inequality holds at least asymptotically for large or small $\alpha.$
	\begin{thm}\label{thm2}
		Let $\Omega_0$ be the square, and let $\Omega$ be any quadrilateral with the same area.  
		Then there exists
		$\alpha_2\leq\alpha_1<0$ such that for all $\alpha\in\left(-\infty,\alpha_2\right]\cup\left[\alpha_1,0\right)$ we have $\lambda_1(\Omega,\alpha)\leq\lambda_1(\Omega_0,\alpha).$
	\end{thm} 
	Note that the constants $\alpha_1,\alpha_2$ depend upon $\Omega.$ This theorem is proved for small $\alpha$ via trial function arguments and for large $\alpha$ via an asymptotic result of Levitin and Parnovski \cite{levitin2008principal}. \\
	Our third result shows that a reverse Faber-Krahn also holds in the case of quadrilaterals far from the square in terms of Hausdorff distance.
	\begin{thm}\label{thm3}
		Let $\Omega_0$ be the square, let $\Omega$ be a quadrilateral with the same area and fix $\alpha<0.$ Then there exists a constant $C$ depending only upon $\alpha$ such that if $d_H(\Omega,\Omega_0)>C$ we have $\lambda_1(\Omega,\alpha)\leq\lambda_1(\Omega_0,\alpha).$
	\end{thm}
	Here $d_H$ is the Hausdorff distance up to isometry.
	This theorem is proved using a test function argument in a similar way to our proof of Theorem \ref{thm2}.\\
	This paper is structured as follows; in Section 2 we provide the preliminary setup required, this involves the parametrisation of an arbitrary quadrilateral and construction of a transform that replaces the Robin problem on this quadrilateral with an isopectral problem on the square. We provide the explicit solution to the Robin problem on the square, and provide useful properties and results that will be used later. In Section 3 we compute the first and second order derivatives of the eigenvalues with respect to the geometric parameters, and prove Theorem \ref{thm1}. In Section 4 we utilise test function arguments to obtain sufficient conditions for $\lambda_1(\Omega)\leq \lambda_1(\Omega_{0}),$ and as a consequence prove Theorem \ref{thm2} and \ref{thm3}. Section 5 is an appendix containing further details of nasty computations required in the proof of Theorem \ref{thm1}.
	\section{Preliminary Setup}
	In this section we follow  Krej\v ci\v r\'ik et al's approach \cite{krejvcivrik2023reverse}. We parametrise an aribtrary quadrilateral of fixed area 2S, then show that the Robin spectral problem on arbitrary quadrilaterals can be reduced to the spectral problem on the square for an associated family of operators dependent upon the parametrisation parameters.
	\subsection{Geometric Setup}
	We provide a parametrisation of a general quadrilateral as follows. 
	Let $c>0,\ a_1,a_2\in\mathbb{R}, \ S_1,S_2>0,$ with $S_1+S_2=2S>0.$ Let $\Omega_{a_1,a_2,c,S_1}$ be the quadrilateral with area $2S$ formed by connecting, in order, the vertices $(-c,0),\ (a_1,b_1),\ (c,0),\ (a_2,-b_2),$ where $b_i=S_i/c.$ The choice $(a_0,c_0)=(0,\sqrt{S})$ leads to the (rotated) square with side length $\sqrt{2S},$ which we shall denote as $\Omega_0,$ when $(a_1,a_2,c,S_1)=(a_0,a_0,c_0,S).$\\
	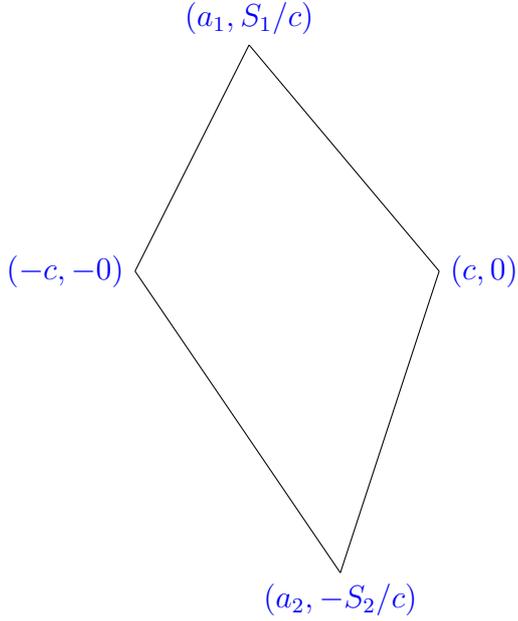
\begin{figure}[H]
		\begin{tikzpicture}
			\draw (2,0) -- (3.5,3) node[anchor=south,blue] {$(a_1,S_1/c$)};
			\draw (3.5,3) -- (6,0) node[anchor=west,blue] {$(c,0$)};
			\draw  (6,0) -- (4.7,-4) node[anchor=north,blue] {$(a_2,-S_2/c$)};
			\draw  (4.7,-4) -- (2,0) node[anchor=east,blue] {$(-c,-0$)};
		\end{tikzpicture}
		\caption{Parametrisation of a Quadrilateral}
	\end{figure}
	We note that by construction the areas of the quadrilaterals are constant and given by
	\begin{align*}
		|\Omega_{a_1,a_2,c,S_1}|=c(b_1+b_2)=S_1+S_2=2S=|\Omega_0|.
	\end{align*} 
	We label the corners of the upper and lower triangles that make up the quadrilateral by
	\begin{align*}
		p_1^{(j)}&=(-c,0),\\
		p_2^{(j)}&=(a_j,(-1)^{j+1}S_j/c),\\
		p_3^{(j)}&=(c,0).
	\end{align*}
	for $j=1,2.$\\
	The boundary of $\Omega_{a_1,a_2,c,S_1}$ may be written as the union of line segments in the form 
	\begin{align*}
		\partial\Omega_{a_1,a_2,c,S_1}=\Gamma_{a_1,c,S_1}^{(1,1)}\cup\Gamma_{a_1,c,S_1}^{(2,1)}\cup\Gamma_{a_2,c,S_2}^{(1,2)}\cup\Gamma_{a_2,c,S_2}^{(2,2)},
	\end{align*}
	where we define $\Gamma_{a_j,c,S_j}^{(i,j)}$ for $i=1,2$ to be the line segment connecting $p^{(j)}_i=(x_{i,j},y_{i,j})$ to $p^{(j)}_{i+1}=(x_{i+1,j},y_{i+1,j}),$ that is
	\begin{align*}
		\Gamma_{a_j,c,S_j}^{(i,j)}=\lbrace p_i^{(j)}+t(p_{i+1}^{(j)}-p_i^{(j)}): t\in[0,1]\rbrace.
	\end{align*}
	\\
	Further, the lengths of the line segments are given by
	\begin{align*}
		|\Gamma_{a_j,c,S_1}|=||p_i^{(j)}-p_{i+1}^{(j)}||=\sqrt{\frac{S_j^2}{c^2}+(a_j+(-1)^{(i+1)}c)^2}.
	\end{align*}
	Note that since $|\Gamma_{0}^{(i,j)}|$ is independent of $(i,j),$ the superscript $(i,j)$ will be omitted. 
	The perimeters of the quadrilaterals are given by
	\begin{align*}
		|\partial\Omega_{a_1,a_2,c,S_1}|=&\Sigma_{i,j}|\Gamma_{a_j,c,S_1}^{(i,j)}|\\
		=&\sqrt{\frac{S_1^2}{c^2}+(a_1-c)^2}+\sqrt{\frac{S_1^2}{c^2}+(a_1+c)^2}+\sqrt{\frac{S_2^2}{c^2}+(a_2-c)^2}+\sqrt{\frac{S_2^2}{c^2}+(a_2+c)^2},\\
	\end{align*}
	with
	\begin{align*}
		\quad |\partial\Omega_0|=4\sqrt{2S}.
	\end{align*}
	We next construct a mapping between the quadrilaterals $\Omega_{0}$ and $\Omega_{a_1,a_2,c,S_1}.$
	\begin{lem}\label{linmap} Let $ a,a_0,S\in\mathbb{R}, c,c_0>0$ with $a\neq\pm c.$ Then, a linear mapping between the triangle $\Omega_0$ defined by the vertices $(c_0,0),(-c_0,0),(a_0,S/c_0)$ and the triangle $\Omega_{a,c}$ defined by the vertices $(c_1,0),(-c_1,0),(a_1,\tilde{S}/c_1)$ is given by
		\begin{align*}
			\mathcal{L}:\Omega_0\to\Omega_{a,c}:(x,y)\mapsto\left(\frac{c_1}{c_0}x+\frac{a_1c_0-a_0c_1}{S}y,\frac{\tilde{S}c_0}{{S}c_1}y\right).
		\end{align*}
	\end{lem}
	\begin{proof}
		
		It is easily verified that $\mathcal{L}$ is a linear mapping that maps the points $(c_0,0),(-c_0,0),(a_0,S/c_0)$ to the images $(c_1,0),(-c_1,0),(a_1,S/c_1).$ Since linear maps preserve line segments, it follows that $\mathcal{L}$ maps $\Omega_0$ to $\Omega_{a,c}.$
	\end{proof}
	In line with Lemma \ref{linmap} above, viewing $\Omega_{a_1,a_2,c,S_1}$ as the union of the triangles $\Omega_{a_1,c}$ and $\Omega_{a_2,c}$ with areas $S_1$ and $S_2$ respectively we construct the following map piecewise
	\begin{align*}
		\mathscr{L}_{a_1,a_2,c,S_1}:\Omega^+_0\to\Omega^+_{a_1,a_2,c,S_1}:\left\lbrace(x,y)\mapsto\left(\frac{c}{c_0}x+\frac{a_1c_0}{S}y,\frac{c_0S_1}{cS}y\right)\right\rbrace,\\
		\mathscr{L}_{a_1,a_2,c,S_1}:\Omega^-_0\to\Omega^-_{a_1,a_2,c,S_1}:\left\lbrace(x,y)\mapsto\left(\frac{c}{c_0}x-\frac{a_2c_0}{S}y,\frac{c_0S_2}{cS}y\right)\right\rbrace,
	\end{align*}
	with inverse
	\begin{align*}
		\mathscr{L}^{-1}_{a_1,a_2,c,S_1}:\Omega^+_{a_1,a_2,c,S_1}\to\Omega^+_{0}:\left\lbrace(x,y)\mapsto\left(\frac{c_0}{c}x-\frac{a_1c_0}{S_1}y,\frac{cS}{c_0S_1}y\right)\right\rbrace,\\
		\mathscr{L}^{-1}_{a_1,a_2,c,S_1
		}:\Omega^-_{a_1,a_2,c,S_1}\to\Omega^-_{0}:\left\lbrace(x,y)\mapsto\left(\frac{c_0}{c}x+\frac{a_2c_0}{S_2}y,\frac{cS}{c_0S_2}y\right)\right\rbrace.
	\end{align*}
	Here, for a set $A\subseteq\mathbb{R}^2,$ we set $A^{\pm}=A\cap\lbrace\pm y\ge 0\rbrace.$\\
	
	We define the maps  $ U_{a_1,a_2,c,S_1}:L^2(\Omega_{a_1,a_2,c,S_1})\to L^2(\Omega_0)$ and 
	$U_{a_1,a_2,c,S_1}^{-1}:L^2(\Omega_0)\to L^2(\Omega_{a_1,a_2,c,S_1})$ via
	\begin{align*}
		U_{a_1,a_2,c,S_1}(u)=&u\circ\mathscr{L}_{a_1,a_2,c,S_1},\\
		U_{a_1,a_2,c,S_1}^{-1}(u)=&u\circ\mathscr{L}_{a_1,a_2,c,S_1}^{-1}.
	\end{align*}
	\begin{lem}\label{mapprop}
		The maps $U_{a_1,a_2,c,S_1},$ $U^{-1}_{a_1,a_2,c,S_1}$ satisfy the following properties:
		\begin{enumerate}[label=(\Roman*)]
			\item For any $u,v\in L^2(\Omega_{a_1,a_2,c,S_1})$ we have  \begin{align*}
				\langle U_{a_1,a_2,c,S_1}\circ u,U_{a_1,a_2,c,S_1}\circ v\rangle_{L^2(\Omega_{0})}=\frac{S}{S_1}\langle u,v\rangle_{L^2(\Omega_{a_1,a_2,c,S_1}^+)}+\frac{S}{S_2}\langle u,v\rangle_{L^2(\Omega_{a_1,a_2,c,S_1}^-)}.
			\end{align*}
			Similarly, for any $\phi,\psi\in L^2(\Omega_{0})$ we have \begin{align*}
				\langle U_{a_1,a_2,c,S_1}^{-1}\circ \phi,U_{a_1,a_2,c,S_1}^{-1}\circ \psi\rangle_{L^2(\Omega_{a_1,a_2,c,S_1})}=\frac{S_1}{S}\langle \phi,\psi\rangle_{L^2(\Omega_{0}^+)}+\frac{S_2}{S}\langle \phi,\psi\rangle_{L^2(\Omega_{0}^-)}.
			\end{align*}
			\item For $i,j=1,2,$ 
			\begin{align*}
				\langle u,v\rangle_{L^2(\Gamma^{(i,j)}_{a_j,c,S_1})}=\frac{|\Gamma_{a_j,c,S_1}^{(i,j)}|}{|\Gamma_{0}|}\langle U_{a_1,a_2,c,S_1}\circ u,U_{a_1,a_2,c,S_1}\circ v\rangle_{L^2(\Gamma^{(i,j)}_{0})},
			\end{align*}
		\end{enumerate}
	\end{lem}
	\begin{proof}   
		\begin{enumerate}[label=(\Roman*)]
			\item 
		Make a change of variables $\tilde{x}=\mathscr{L}_{a_1,a_2,c,S_1}(x),$ where the Jacobian is
		\begin{align*}
			J=\det \left[\begin{matrix}
				\frac{c_0}{c} & {c_0}\left(-\frac{a_1}{S_1}\bchi_{y\geq0}+\frac{a_2}{S_2}\bchi_{y<0}\right)\\
				0 & \frac{cS}{c_0\left(S_1\bchi_{y\geq0}+S_2\bchi_{y<0}\right)}
			\end{matrix}\right]=\frac{S}{\left(S_1\bchi_{y\geq0}+S_2\bchi_{y<0}\right)},
		\end{align*}
		and the proof follows.\\
		The proof for $U_{a_1,a_2,c,S_1}^{-1}$  is analogous.

		\item Let $q_i^j=(\xi_i^j,\eta_i^j)=\mathscr{L}_{a_1,a_2,c,S_1}(p_i^j).$ Let $(x_i^j,y_i^j)=p_i^j\in\Omega_{a_1,a_2,c,S_1}$ such that $\Gamma_{a_j,c,S_1}^{(i,j)}$ is the line segment connecting $p_i^j$ and $p_{i+1}^j.$ Further, let $\psi=U_{a_1,a_2,c,S_1}\circ u$ and $\phi =U_{a_1,a_2,c,S_1}\circ v.$ Then,
		\begin{align*}
			\frac{\langle u,v\rangle_{L^2(\Gamma^{(i,j)}_{a_j,c,S_1})}}{\langle \psi,\phi\rangle_{L^2(\Gamma^{(i,j)}_{0})}}=&\frac{|\Gamma_{a_j,c,S_1}^{(i,j)}|\int_{0}^{1}u\left(\left[\begin{matrix}x_i^j\\y_i^j\end{matrix}\right]+t\left[\begin{matrix}x_{i+1}^j-x_i^j\\y_{i+1}^j-y_i^j\end{matrix}\right]\right)v\left(\left[\begin{matrix}x_i^j\\y_i^j\end{matrix}\right]+t\left[\begin{matrix}x_{i+1}^j-x_i^j\\y_{i+1}^j-y_i^j\end{matrix}\right]\right)dt}{|\Gamma_{0}|\int_{0}^{1}\psi\left(\left[\begin{matrix}\xi_i^j\\ \eta_i^j\end{matrix}\right]+\tau\left[\begin{matrix}\xi_{i+1}^j-\xi_i^j\\ \eta_{i+1}^j-\eta_i^j\end{matrix}\right]\right)\phi\left(\left[\begin{matrix}\xi_i^j\\ \eta_i^j\end{matrix}\right]+\tau\left[\begin{matrix}\xi_{i+1}^j-\xi_i^j\\ \eta_{i+1}^j-\eta_i^j\end{matrix}\right]\right)d\tau},\\
			=&\frac{|\Gamma_{a_j,c,S_1}^{(i,j)}|}{|\Gamma_{0}|},
		\end{align*}
		where we have used the fact that $\psi=u\circ\mathscr{L}_{a_1,a_2,c,S_1},$ $\phi=v\circ\mathscr{L}_{a_1,a_2,c,S_1},$ and $\mathscr{L}_{a_1,a_2,c,S_1}$ is linear, so 
		\begin{align*}
			\mathscr{L}_{a_1,a_2,c,S_1}\left(\left[\begin{matrix}\xi_i^j\\ \eta_i^j\end{matrix}\right]+\tau\left[\begin{matrix}\xi_{i+1}^j-\xi_i^j\\ \eta_{i+1}^j-\eta_i^j\end{matrix}\right]\right)=&\mathscr{L}_{a_1,a_2,c,S_1}\left(\left[\begin{matrix}\xi_i^j\\ \eta_i^j\end{matrix}\right]\right)+\tau\mathscr{L}_{a_1,a_2,c,S_1}\left(\left[\begin{matrix}\xi_{i+1}^j-\xi_i^j\\ \eta_{i+1}^j-\eta_i^j\end{matrix}\right]\right),\\
			=&\left[\begin{matrix}x_i^j\\y_i^j\end{matrix}\right]+\tau\left[\begin{matrix}x_{i+1}^j-x_i^j\\y_{i+1}^j-y_i^j\end{matrix}\right].
		\end{align*}
	\end{enumerate}
\end{proof}
\subsection{Functional Setup}
We view the Robin problem (\ref{Eq:robA}-\ref{Eq:robB}) on the quadrilateral $\Omega_{a_1,a_2,c,S_1}$ as the spectral problem in the Hilbert space $L^2(\Omega)$ for
\begin{align}\label{opp}
	H_{\alpha}u=-\Delta u, \quad D(H_\alpha)=\lbrace u\in H^1(\Omega) \ : \ \frac{\partial u}{\partial n}+\alpha u=0 \ \text{ on } \ \partial\Omega\rbrace,
\end{align}
where the derivative is viewed as a distribution where necessary, and the associated quadratic form is
\begin{align*}
	h_{\alpha}[u]=\int_{\Omega}|\nabla u|^2+\alpha\int_{\partial\Omega}|u|^2, \quad D(h_\alpha)=H^{1}(\Omega).
\end{align*}
We define the operator $\hat{H}_{\alpha,a_1,a_2,c,S_1}$ via
\begin{align*}
	\hat{H}_{\alpha,a_1,a_2,c,S_1}=U_{a_1,a_2,c,S_1}H_\alpha U_{a_1,a_2,c,S_1}^{-1},
\end{align*}
where $D(\hat{H}_{\alpha,a_1,a_2,c,S_1})=U_{a_1,a_2,c,S_1}(D(H_{\alpha})).$\\
Then the operators $H_\alpha$ and $\hat{H}_{\alpha,a_1,a_2,c,S_1}$ are isospectral. One observes that $\hat{H}_{\alpha,a_1,a_2,c}$ is the operator on $L^2(\Omega_0)$ associated with the quadratic form
\begin{align*}
	\hat{h}_{\alpha,a_1,a_2,c,S_1}[\psi]=h_\alpha\left[\sqrt{\frac{S}{S_1}\bchi_{y>0}+\frac{S}{S_2}\bchi_{y<0}}U^{-1}_{a_1,a_2,c,S_1}\psi\right].\end{align*}
This associated form can be written
\begin{align}
	\hat{h}_{\alpha,a_1,a_2,c,S_1}[\psi,\varphi]\nonumber
	=&\int_{\Omega_{0}}\left(\left({\partial_1 \mathscr{L}_{a_1,a_2,c,S_1}^{-1}}\cdot \nabla\right)\psi\right)\left(\left({\partial_1 \mathscr{L}_{a_1,a_2,c,S_1}^{-1}}\cdot \nabla\right)\varphi\right)\nonumber\\
	+&\left(\left({\partial_2 \mathscr{L}_{a_1,a_2,c,S_1}^{-1}}\cdot \nabla\right)\psi\right)\left(\left({\partial_2 \mathscr{L}_{a_1,a_2,c,S_1}^{-1}}\cdot \nabla\right)\varphi\right) d{x}d{y}\nonumber\\
	+&\alpha\sum_{i,j}\frac{|\Gamma_{a_j,c,S_1}^{(i,j)}|S}{|\Gamma_{0}|S_j}\langle\psi,\varphi\rangle_{L^2(\Gamma^{(i,j)}_{0})},\label{eq:hnodif}
\end{align}\\
or more explicitly,
\begin{align}\label{explexpl}
	\hat{h}_{\alpha,a_1,a_2,c,S_1}[\psi]=&\frac{c_0^2}{c^2}||\partial_1\psi||^2_{L^2(\Omega_0)}\nonumber\\+&\left|\left|\left(\bchi_{y<0}\frac{a_2}{S_2}-\bchi_{y>0}\frac{a_1}{S_1}\right)c_0\partial_1\psi+\frac{cS}{c_0\left(S_1\bchi_{y>0}+S_2\bchi_{y<0}\right)}\partial_2\psi\right|\right|^2_{L^2(\Omega_0)}\nonumber\\
	&+\frac{\alpha S}{\sqrt{2}cc_0}\sum_{i,j}\frac{\sqrt{S_j^2+c^2(a_j+(-1)^{i+1}c)^2}}{S_j}||\psi||^2_{L^2(\Gamma^{(i,j)}_{0})}, 
\end{align}
where we write $\partial_1$  to mean $\frac{\partial}{\partial x}$ and $\partial_2$  to mean $\frac{\partial}{\partial y}.$\\
In all that follows, let $\psi_{a_1,a_2,c,S_1}$ be the positive first eigenfunction of $\hat{H}_{\alpha,a_1,a_2,c,S_1}$ corresponding to $\lambda_{a_1,a_2,c,S_1},$ scaled so that $||\psi_{a_1,a_2,c,S_1}||_{L^2(\Omega_0)}=1.$ Then the eigenvalue problem $\hat{H}_{\alpha,a_1,a_2,c,S_1}\psi_{a_1,a_2,c}=\lambda_{a_1,a_2,c,S_1}\psi_{a_1,a_2,c,S_1}$ is equivalent to the weak problem 
\begin{align}\label{weakform}
	\hat{h}_{\alpha,a_1,a_2,c,S_1}[\phi,\psi_{a_1,a_2,c,S_1}]=\lambda_{a_1,a_2,c,S_1}\langle\phi,\psi_{a_1,a_2,c,S_1}\rangle, \quad \forall \phi\in H^1(\Omega_{0}).
\end{align}
Further, both $\lambda_{a_1,a_2,c,S_1}$ and $\psi_{a_1,a_2,c,S_1}$ are real analytic functions of the parameters $a_1,a_2,c,S_1$ for $a_1,a_2\in\mathbb{R},$ $c\in (0,\infty)$ and $S_1\in(0,2S).$ This may be observed by noting that $\hat{H}_{\alpha,a_1,a_2,c,S_1}$ forms a holomorphic family of type (B) in the sense of Kato  (see \cite[Theorem 4.8  section VII]{KatoTosio1995PTfL}).\\
Note that for the remainder of the paper we will use $\lambda_{0}$ and $\psi_0$ to denote $\lambda_{a_1,a_2,c,S_1}$ and $\psi_{a_1,a_2,c,S_1}$ when $(a_1,a_2,c,S_1)=(0,0,c_0,S).$

\subsection{The Solution on the Square}
The solution of the problem $(\ref{opp})$ is known explicitly for rectangles (see for instance \cite{laugesen2019robin}). In particular, the explicit solution on the square and its properties will be of use in the following sections, and so for completeness we state its solution below, as well as some of its properties which will be of use later. We will frame the problem in terms of the rotated square $\Omega_{0}$ of area $2S,$ for $S>0,$ with vertices $(\pm\sqrt{S},0)$ and $(0,\pm\sqrt{S}).$
\begin{figure}[H]
	\begin{tikzpicture}
		\draw (2,0) -- (4,2) node[anchor=south,blue] {$(0,\sqrt{S})$};
		\draw (4,2) -- (6,0) node[anchor=west,blue] {$(\sqrt{S},0)$};
		\draw  (6,0) -- (4,-2) node[anchor=north,blue] {$(0,-\sqrt{S})$};
		\draw  (4,-2) -- (2,0) node[anchor=east,blue] {$(-\sqrt{S},0)$};
	\end{tikzpicture}
	\caption{The Rotated Square $\Omega_0$}
\end{figure}
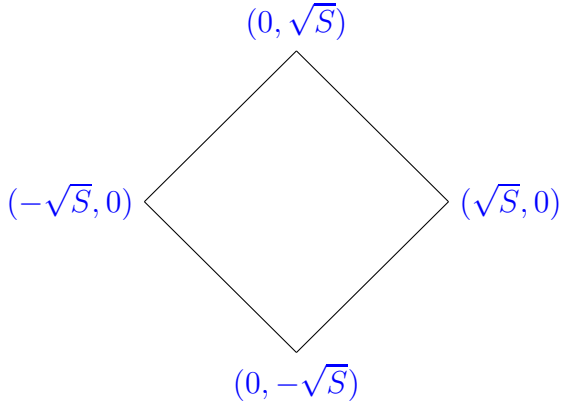
\begin{prop}\label{pde}
	Consider the eigenvalue problem 
	\begin{align}\label{eigpro}
		\begin{cases}
			\Delta u =-\lambda u,\quad \text{in } \Omega_0,\\
			\frac{\partial u}{\partial \nu}+\alpha u=0 \quad \text{on } \partial\Omega_0.
		\end{cases}
	\end{align}
	Then the eigenvalue problem (\ref{eigpro}) has first eigenfunction $u_1$ given by
	\begin{align}\label{explform}
		\begin{cases}
			u_1=\cos\left({\frac{\sqrt{\lambda_1}}{2}} (x+y)\right)\cos\left({\frac{\sqrt{\lambda_1}}{2}} (y-x)\right),\quad \alpha>0,\\
			u_1=\cosh\left({\frac{\sqrt{-\lambda_1}}{2}} (x+y)\right)\cosh\left({\frac{\sqrt{-\lambda_1}}{2}} (y-x)\right),\quad \alpha<0,
		\end{cases}
	\end{align}
	with corresponding first eigenvalues $\lambda_1$ given by 
	\begin{align*}
		\lambda_1=\begin{cases}
			2\left(\frac{f^{-1}(\alpha L)}{L}\right)^2,\quad \alpha>0,\\
			-2\left(\frac{g^{-1}(-\alpha L)}{L}\right)^2,\quad \alpha<0,
		\end{cases}
	\end{align*}
	where 
	\begin{align*}
		f(t)=&t\tan(t), \quad t\in(0,\pi/2)\cup(\pi/2,\pi),\\
		g(t)=&t\tanh(t), \quad t\geq 0,
	\end{align*}
	and $L=\sqrt{S/2}.$
\end{prop}
\begin{cry}\label{symprop}
	Let $\alpha\neq0,$ then the function $u_1$ as given by Proposition \ref{pde} has the following properties
	\begin{enumerate}[label=(\Roman*)]
		\item $u_1$ is even in $x$
		\item $u_1$ is even in $y$
		\item $||u_1||_{L^2(\Gamma_{0}^{i,j})}$ is independent of $i,j$
		\item $||\partial_1u_1||_{L^2(\Omega_0)}^2$=$||\partial_2u_1||_{L^2(\Omega_0)}^2=\frac{1}{2}||\nabla u_1||_{L^2(\Omega_0)}^2$
		\item $\langle \partial_1 u_1,\partial_2 u_1\rangle_{L^2(\Omega_0^+)}=0$
		\item $\langle \partial_1 u_1,\partial_2 u_1\rangle_{L^2(\Omega_0^-)}=0$
		\item $||\partial_1u_1||^2_{L^2(\Omega_0^+)}=||\partial_1u_1||^2_{L^2(\Omega_0^-)}=\frac{1}{2}||\partial_1u_1||^2_{L^2(\Omega_0)}$
		\item $||\partial_2u_1||^2_{L^2(\Omega_0^+)}=||\partial_2u_1||^2_{L^2(\Omega_0^-)}=\frac{1}{2}||\partial_2u_1||^2_{L^2(\Omega_0)}.$
	\end{enumerate}
\end{cry}
\begin{proof}
	Follows from checking the explicit form (\ref{explform}).
\end{proof}
In the following lemmas we provide a useful result that will be of use in the following section.                                                                      
\begin{lem}\label{extre}
	Let \begin{align*}
		\zeta(\alpha)=||\nabla \psi_0||_{L^2(\Omega_0)}^2+\alpha C||\psi_0||_{L^2(\partial \Omega_0)}^2,
	\end{align*}
	then $\zeta(\alpha)<0$ for all $\alpha<0,$ provided that $C\in(0,1).$
\end{lem}                                                                                                
\begin{proof}
	
	Since
	\begin{align*}
		\lambda_0=||\nabla \psi_0||_{L^2(\Omega_0)}^2+\alpha ||\psi_0||_{L^2(\partial \Omega_0)}^2,
	\end{align*}
	it follows that 
	\begin{align*}
		\zeta(\alpha)=\lambda_{0}+\alpha(C-1)||\psi_0||_{L^2(\partial \Omega_0)}^2.
	\end{align*}
	Let $(\lambda_0)'_{\alpha}$ denote $\frac{\partial\lambda_{0}}{\partial \alpha}.$ Then
	\begin{align*}
		(\lambda_0)'_{\alpha}=||\psi_0||_{L^2(\partial \Omega_0)}^2,
	\end{align*}
	(see for instance \cite[section 4.3.2]{henrot2017shape} ) 
	and so
	\begin{align*}
		\zeta(\alpha)=\lambda_{0}+\alpha(C-1)(\lambda_0)'_{\alpha}.
	\end{align*}
	Thus
	\begin{align*}
		\zeta(\alpha)<0
		\iff \frac{(\lambda_0)'_{\alpha}}{\lambda_{0}}>\frac{1}{\alpha(1-C)}.
	\end{align*}
	Recall from Proposition \ref{pde}
	\begin{align*}
		g\left(L\sqrt{\frac{-\lambda_{0}}{2}}\right)=-\alpha L,
	\end{align*}
	where $g(t)=t\tanh (t).$\\
	Then
	\begin{align*}
		g'\left(L\sqrt{\frac{-\lambda_{0}}{2}}\right)\frac{1}{\sqrt{\frac{-\lambda_{0}}{2}}}\frac{-L}{2}(\lambda_0)'_{\alpha}=-L
		\implies (\lambda_0)'_{\alpha}=\frac{-L\lambda_0}{L\sqrt{\frac{-\lambda_{0}}{2}}g'\left(L\sqrt{\frac{-\lambda_{0}}{2}}\right)}.
	\end{align*}
	Now,
	
	\begin{align*}
		\frac{(\lambda_0)'_{\alpha}}{\lambda_{0}}=\frac{-L}{\left(L\sqrt{\frac{-\lambda_{0}}{2}}\right)\tanh\left(L\sqrt{\frac{-\lambda_{0}}{2}}\right)+\left(L\sqrt{\frac{-\lambda_{0}}{2}}\right)^2\text{sech}^2\left(L\sqrt{\frac{-\lambda_{0}}{2}}\right)},
	\end{align*}
	where we have used \begin{align*}
		g'(t)=\tanh(t)+t\text{sech}^2(t).
	\end{align*}
	Thus, 
	\begin{align*}
		&\zeta(\alpha)<0,\\
		\iff& \frac{-L}{\left(L\sqrt{\frac{-\lambda_{0}}{2}}\right)\tanh\left(L\sqrt{\frac{-\lambda_{0}}{2}}\right)+\left(L\sqrt{\frac{-\lambda_{0}}{2}}\right)^2\text{sech}^2\left(L\sqrt{\frac{-\lambda_{0}}{2}}\right)}>\frac{1}{\alpha(1-C)},\\
		\iff& \frac{\left(L\sqrt{\frac{-\lambda_{0}}{2}}\right)\tanh\left(\sqrt{L\frac{-\lambda_{0}}{2}}\right)(1-C)}{\left(L\sqrt{\frac{-\lambda_{0}}{2}}\right)\tanh\left(L\sqrt{\frac{-\lambda_{0}}{2}}\right)+\left(L\sqrt{\frac{-\lambda_{0}}{2}}\right)^2\text{sech}^2\left(L\sqrt{\frac{-\lambda_{0}}{2}}\right)}<1.
	\end{align*}
	This is always satisfied, as the lefthand side is always less than $1.$
\end{proof}                                               
\section{A Local Optimisation Result} 
The goal of this section is to show that there is a local maxima $\lambda_{a_1,a_2,c,S_1}$ at $(a_1,a_2,c,S_1)=(a_0,a_0,c_0,S).$ 
We use the expression (\ref{weakform}) for the weak form $\hat{h}_{\alpha,a_1,a_2,c,S_1},$ to obtain expressions for $\lambda_{a_1,a_2,c,S_1}$ and its derivatives. These derivatives allow us to find conditions on the Hessian matrix for $\lambda_{a_1,a_2,c,S_1}$ corresponding to a local maximum at $(a_1,a_2,c,S_1)=(a_0,a_0,c_0,S).$\\
\subsection{First Derivatives}
We introduce the following notation; let $v$ be any of the perturbation parameters $a_1,a_2,c$ or $S_1$, and let $f^v$ denote $\frac{\partial f}{\partial v}.$ Further define as follows 
\begin{align}\label{hdif}
	\bar{h}_{\alpha,a_1,a_2,c,S_1}^{d_v}[f,g]
	=&\sum_{i=1}^2\left[\left\langle\left(\frac{\partial}{\partial v}\left({\partial_i \mathscr{L}_{a_1,a_2,c,S_1}^{-1}}\right)\cdot \nabla\right)f,\left({\partial_i \mathscr{L}_{a_1,a_2,c,S_1}^{-1}}\cdot \nabla\right)g\right\rangle_{L^2(\Omega_0)}\right.\nonumber\\ &+\left.\left\langle\left({\partial_i \mathscr{L}_{a_1,a_2,c,S_1}^{-1}}\cdot \nabla\right)f,\left(\frac{\partial}{\partial v}\left({\partial_i \mathscr{L}_{a_1,a_2,c,S_1}^{-1}}\right)\cdot \nabla\right)g\right\rangle_{L^2(\Omega_0)}\right],\nonumber\\
	\tilde{h}_{\alpha,a_1,a_2,c,S_1}^{d_v}[f,g]=&\frac{S}{|\Gamma_{0}|}\sum_{i,j}\frac{\partial}{\partial v}\left(\frac{|\Gamma_{a_j,c,S_1}^{(i,j)}}{S_j}|\right)\langle f,g\rangle_{L^2(\Gamma^{(i,j)}_{0})},\nonumber\\
	\hat{h}_{\alpha,a_1,a_2,c,S_1}^{d_v}[f,g]=&\bar{h}_{\alpha,a_1,a_2,c,S_1}^{d_v}[f,g]+\alpha\tilde{h}_{\alpha,a_1,a_2,c,S_1}^{d_v}[f,g].
\end{align}
Then it follows from (\ref{eq:hnodif}) that 
\begin{align*}
	\frac{\partial}{\partial v}\left(\hat{h}_{\alpha,a_1,a_2,c,S_1}[\phi,\psi_{a_1,a_2,c,S}]\right)=\hat{h}^{d_v}_{\alpha,a_1,a_2,c,S_1}[\phi,\psi_{a_1,a_2,c,S_1}]+\hat{h}_{\alpha,a_1,a_2,c,S_1}[\phi,\psi^v_{a_1,a_2,c,S_1}].
\end{align*}
Then, observing from (\ref{weakform}) that
\begin{align*}
	\frac{\partial}{\partial v}\left(\lambda_{a_1,a_2,c,S_1}\langle\phi,\psi_{a_1,a_2,c,S_1}\rangle_{L^2(\Omega_0)}\right)=\frac{\partial}{\partial v}\left(\hat{h}_{\alpha,a_1,a_2,c,S_1}[\phi,\psi_{a_1,a_2,c,S_1}]\right),
\end{align*}
it then follows that
\begin{align}
	\label{lamfulldiv}
	\lambda^v_{a_1,a_2,c,S_1}\langle\phi,\psi_{a_1,a_2,c,S_1}\rangle_{L^2(\Omega_0)}+\lambda_{a_1,a_2,c,S_1}\langle\phi,\psi^v_{a_1,a_2,c,S_1}\rangle_{L^2(\Omega_0)}\nonumber\\=\hat{h}^{d_v}_{\alpha,a_1,a_2,c,S_1}[\phi,\psi_{a_1,a_2,c,S_1}]+\hat{h}_{\alpha,a_1,a_2,c,S_1}[\phi,\psi^v_{a_1,a_2,c,S_1}],
\end{align}
and so with $\phi=\psi_{a_1,a_2,c,S_1},$ we readily obtain
\begin{align}\label{lamdiv}
	{\lambda}^v_{a_1,a_2,c,S_1}=\hat{h}^{d_v}_{a_1,a_2,c,S_1}[\psi_{a_1,a_2,c,S_1}].
\end{align}
\begin{lem}\label{lamzero}
	$\lambda_0^v=0$ for $v=a_1,a_2,c,S_1.$
\end{lem}
\begin{proof}
	Step 1. When $(a_1,a_2,c,S_1)=(0,0,c_0,S)$ the formulae (\ref{hdif}) simplify as follows (see Lemma \ref{hd1v2} in the appendix for intermediary details)
	\begin{align*}
		\hat{h}_{\alpha,a_1,a_2,c}^{d_{a_1}}[\psi_{0}]=&-\frac{2c_0}{S}\langle\partial_1\psi_{0},\partial_2\psi_{0}\rangle_{L^2(\Omega_{0}^+)}\\
		+&\alpha\frac{c_0}{|\Gamma_{0}|^2}\left(||\psi_0||^2_{L^2(\Gamma_{0}^{(1,1)})}-||\psi_{0}||^2_{L^2(\Gamma_{0}^{(2,1)})}\right)\\
		\hat{h}_{\alpha,a_1,a_2,c}^{d_{a_2}}[\psi_{0}]=&2\frac{c_0}{S}\langle\partial_1\psi_{0},\partial_2\psi_{0}\rangle_{L^2(\Omega_{0}^-)}\\
		+&\alpha\frac{c_0}{|\Gamma_{0}|^2}\left(||\psi_{0}||^2_{L^2(\Gamma_{0}^{(1,2)})}-||\psi_{0}||^2_{L^2(\Gamma_{0}^{(2,2)})}\right)\\
		\hat{h}_{\alpha,a_1,a_2,c}^{d_{c}}[\psi_{0}]=&
		\frac{2}{c_0}\left(||\partial_2\psi_{0}||^2_{L^2(\Omega_{0})}-||\partial_1\psi_{0}||_{L^2(\Omega_0)}\right)\\
		+&\alpha\frac{-\frac{S^2}{c_0^3}+c_0}{|\Gamma_{0}|^2}||\psi_{0}||^2_{L^2(\partial\Omega_0)}\\
		\hat{h}_{\alpha,a_1,a_2,c}^{d_{S_1}}[\psi_{0}]=&\frac{2}{S}\left(||\partial_2\psi_{0}||^2_{L^2(\Omega_0^-)}-||\partial_2\psi_{0}||^2_{L^2(\Omega_0^+)}\right)\\
		+&\frac{\alpha}{|\Gamma_0|}\left(\frac{S}{|\Gamma_0|c_0^2}-\frac{|\Gamma_0|}{S}\right)\left(||\psi_{0}||^2_{L^2(\Gamma_{0})}+||\psi_{0}||^2_{L^2(\Gamma_{0})}\right.\\
		-&\left.||\psi_{0}||^2_{L^2(\Gamma_{0}^{(1,2)})}-||\psi_{0}||^2_{L^2(\Gamma_{0}^{(2,2)})}\right).
	\end{align*}
	Step 2. Recall from Corollary \ref{symprop} the following symmetry properties of $\psi_0:$\\
	$\langle\partial_1\psi_{0},\partial_2\psi_{0}\rangle_{L^2(\Omega_0^+)}=0,$\\ 
	$\langle\partial_1\psi_{0},\partial_2\psi_{0}\rangle_{L^2(\Omega_0^-)}=0,$  \\                                                                 
	$||\psi_{0}||_{L^2(\Gamma_{0}^{i,1})}=||\psi_{0}||_{L^2(\Gamma_{0}^{j,1})},$ \\
	$||\psi_{0}||_{L^2(\Gamma_{0}^{1,i})}=||\psi_{0}||_{L^2(\Gamma_{0}^{1,j})},$ \\
	$||\partial_1\psi_{0}||_{L^2(\Omega_0)}=||\partial_2\psi_{0}||_{L^2(\Omega_0)},$ \\
	$||\partial_2\psi_{0}||_{L^2(\Omega_0^+)}=||\partial_2\psi_{0}||_{L^2(\Omega_0^-)}.$ \\
	Step 3. Applying the symmetry properties from step 2 to the formulae in step 1, recalling that $S^2/c_0^3=c_0$ and applying equation (\ref{lamdiv}) leads to the desired result.                                                         \end{proof}
\subsection{Second Derivatives}
We introduce the following notation.
Let\begin{align}\label{dd1}
\hat{h}_{\alpha,a_1,a_2,c,S_1}^{d_{v_1},d_{v_2}}[\psi_{a_1,a_2,c,S_1},\psi_{a_1,a_2,c,S_1}]=&\bar{h}_{\alpha,a_1,a_2,c,S_1}^{d_{v_1},d_{v_2}}[\psi_{a_1,a_2,c,S_1},\psi_{a_1,a_2,c,S_1}]\nonumber\\
+&\tilde{h}_{\alpha,a_1,a_2,c,S_1}^{d_{v_1},d_{v_2}}[\psi_{a_1,a_2,c,S_1},\psi_{a_1,a_2,c,S_1}],
\end{align}
where 
\begin{align}\label{dd2}
\bar{h}^{d_{v_1}d_{v_2}}_{\alpha,a_1,a_2,c,S_1}[f,g]=&\sum_{i=1}^2\left[\left\langle\left(\frac{\partial^2}{\partial v_1\partial v_2}\left({\partial_i \mathscr{L}_{a_1,a_2,c,S_1}^{-1}}\right)\cdot \nabla\right)f,\left({\partial_i \mathscr{L}_{a_1,a_2,c,S_1}^{-1}}\cdot \nabla\right)g\right\rangle_{L^2(\Omega_0)}\right.\nonumber\\
+&\left.\left\langle\left(\frac{\partial}{\partial v_1}\left({\partial_i \mathscr{L}_{a_1,a_2,c,S_1}^{-1}}\right)\cdot \nabla\right)f,\left(\frac{\partial}{\partial v_2}\left({\partial_i \mathscr{L}_{a_1,a_2,c,S_1}^{-1}}\right)\cdot \nabla\right)g\right\rangle_{L^2(\Omega_0)}\right.\nonumber\\
+&\left.\left\langle\left(\frac{\partial}{\partial v_2}\left({\partial_i \mathscr{L}_{a_1,a_2,c,S_1}^{-1}}\right)\cdot \nabla\right)f,\left(\frac{\partial}{\partial v_1}\left({\partial_i \mathscr{L}_{a_1,a_2,c,S_1}^{-1}}\right)\cdot \nabla\right)g\right\rangle_{L^2(\Omega_0)}\right.\nonumber\\ +&\left.\left\langle\left({\partial_i \mathscr{L}_{a_1,a_2,c,S_1}^{-1}}\cdot \nabla\right)f,\left(\frac{\partial^2}{\partial v_1\partial v_2}\left({\partial_i \mathscr{L}_{a_1,a_2,c,S_1}^{-1}}\right)\cdot \nabla\right)g\right\rangle_{L^2(\Omega_0)}\right],
\end{align}
and
\begin{align}\label{dd3}
\tilde{h}^{d_{v_1},d_{v_2}}_{\alpha,a_1,a_2,c,S_1}[f,g]=\frac{S}{|\Gamma_{0}|}\sum_{i,j}\frac{\partial^2}{\partial v_1\partial v_2}\left(\frac{|\Gamma_{a_j,c,S_1}^{(i,j)}|}{S_j}\right)\langle f,g\rangle_{L^2(\Gamma^{(i)}_{0})}.
\end{align}
Then, differentiating (\ref{lamdiv}), we have
\begin{align}\label{lamdiv2}
{\lambda}^{v_1,v_2}_{a_1,a_2,c,S_1}=2\hat{h}_{a_1,a_2,c,S_1}^{d_{v_1}}[\psi_{a_1,a_2,c,S_1},\psi_{a_1,a_2,c,S_1}^{v_2}]+\hat{h}_{a_1,a_2,c,S_1}^{d_{v_1},d_{v_2}}[\psi_{a_1,a_2,c,S_1},\psi_{a_1,a_2,c,S_1}].
\end{align}

\begin{lem}\label{eigprop}
The following are true:
\begin{enumerate}[label=(\Roman*)]
	\item
	\begin{align*}
		{\lambda}^{v_1,v_2}_{a_1,a_2,c,S_1}=&2(\lambda_{a_1,a_2,c,S_1}^{v_1}\langle\psi_{a_1,a_2,c,S_1},\psi^{v_2}_{a_1,a_2,c,S_1}\rangle+\lambda_{a_1,a_2,c,S_1}\langle\psi_{a_1,a_2,c,S_1}^{v_1},\psi^{v_2}_{a_1,a_2,c,S_1}\rangle\\-&\hat{h}_{\alpha,a_1,a_2,c,S_1}[\psi^{v_1}_{a_1,a_2,c,S_1},\psi^{v_2}_{a_1,a_2,c,S_1}])
		+\hat{h}_{\alpha,a_1,a_2,c,S_1}^{d_{v_1},d_{v_2}}[\psi_{a_1,a_2,c,S_1},\psi_{a_1,a_2,c,S_1}],
	\end{align*}
	\item 
	\begin{align*}
		{\lambda}^{v_1,v_2}_{0}=&2(\lambda_{0}\langle\psi_{0}^{v_1},\psi^{v_2}_{0}\rangle-\hat{h}_{\alpha}[\psi^{v_1}_{0},\psi^{v_2}_{0}])+\hat{h}_{\alpha}^{d_{v_1},d_{v_2}}[\psi_{0},\psi_{0}],
	\end{align*}
	\item 
	\begin{align*}
		{\lambda}^{v,v}_{0}=&2(\lambda_{0}||\psi_{0}^{v}||^2_{L^2(\Omega_0)}-\hat{h}_{\alpha}[\psi^{v}_{0}])+\hat{h}_{\alpha}^{d_{v},d_{v}}[\psi_{0}],
	\end{align*}
	\item If $\hat{h}_{\alpha}^{d_{v},d_{v}}[\psi_{0}]<0,$ then ${\lambda}^{v,v}_{0}<0.$
\end{enumerate}
\end{lem} 
\begin{proof}
\begin{enumerate}[label=(\Roman*)]
	\item Recall from (\ref{lamfulldiv}) that 
	\begin{align*}
		&\lambda_{a_1,a_2,c,S_1}^{v_1}\langle\psi_{a_1,a_2,c,S_1},\phi\rangle+\lambda_{a_1,a_2,c,S_1}\langle\psi_{a_1,a_2,c,S_1}^{v_1},\phi\rangle\\
		=&\hat{h}^{d_{v_1}}_{\alpha,a_1,a_2,c,S_1}[\phi,\psi_{a_1,a_2,c,S_1}]+\hat{h}_{\alpha,a_1,a_2,c,S_1}[\phi,\psi^{v_1}_{a_1,a_2,c,S_1}],
	\end{align*}
	and so setting $\phi=\psi^{v_2}_{a_1,a_2,c_0,S_1}$ we have
	\begin{align*}
		&\lambda_{a_1,a_2,c,S_1}^{v_1}\langle\psi_{a_1,a_2,c,S_1},\psi^{v_2}_{a_1,a_2,c,S_1}\rangle+\lambda_{a_1,a_2,c,S_1}\langle\psi_{a_1,a_2,c,S_1}^{v_1},\psi^{v_2}_{a_1,a_2,c}\rangle\\
		=&\hat{h}^{d_{v_1}}_{\alpha,a_1,a_2,c,S_1}[\psi^{v_2}_{a_1,a_2,c,S_1},\psi_{a_1,a_2,c,S_1}]+\hat{h}_{\alpha,a_1,a_2,c,S_1}[\psi^{v_1}_{a_1,a_2,c,S_1},\psi^{v_2}_{a_1,a_2,c,S_1}].
	\end{align*}
	Hence, using (\ref{lamdiv2}) we have
	\begin{align*}
		{\lambda}^{v_1,v_2}_{a_1,a_2,c,S_1}=&2(\lambda_{a_1,a_2,c,S_1}^{v_1}\langle\psi_{a_1,a_2,c,S_1},\psi^{v_2}_{a_1,a_2,c,S_1}\rangle+\lambda_{a_1,a_2,c,S_1}\langle\psi_{a_1,a_2,c,S_1}^{v_1},\psi^{v_2}_{a_1,a_2,c,S_1}\rangle\\-&\hat{h}_{\alpha,a_1,a_2,c,S_1}[\psi^{v_1}_{a_1,a_2,c,S_1},\psi^{v_2}_{a_1,a_2,c,S_1}])
		+\hat{h}_{\alpha,a_1,a_2,c}^{d_{v_1},d_{v_2}}[\psi_{a_1,a_2,c},\psi_{a_1,a_2,c}],
	\end{align*}\\
	\item Follows from recalling that $\lambda_{0}^{v_1}=0,$ and applying (I).\\
	\item Apply (II) with the choice $v_1=v_2=v$.\\
	\item Follows from (III) once one recalls, using the variational characterisation (\ref{eq:eig1}), that $\hat{h}_\alpha{[\psi_{0}^v]}\ge\lambda_{0}||\psi_{0}^v||^2_{L^2(\Omega_0)}.$ \\
\end{enumerate}
\end{proof}                                      
\begin{lem}\label{ddexplicit}
Explicitly, we have  for $j=1,2$                                        \begin{enumerate}[label=(\Roman*)]
	\item 
	\begin{align*}
		\lambda_0^{a_j,a_j}=&2\left(\lambda_{0}||\psi_{0}^{a_j}||^2_{L^2(\Omega_0)}-\hat{h}_{\alpha}[\psi^{a_j}_{0}]\right)+\frac{1}{2S}||\nabla\psi_{0}||^2_{L^2(\Omega_0)}+\alpha\frac{1}{8S}||\psi_{0}||^2_{L^2(\partial\Omega_0)}
	\end{align*}
	\item
	\begin{align*}
		\lambda_{0}^{c,c}=&2\left(\lambda_{0}||\psi_{0}^{c}||^2_{L^2(\Omega_0)}-\hat{h}_{\alpha}[\psi^{c}_{0}]\right)+\frac{4}{S}||\nabla\psi_{0}||^2_{L^2(\Omega_0)}+\alpha\frac{2}{S}||\psi_{0}||^2_{L^2(\partial\Omega_{0})}
	\end{align*}
	\item 
	\begin{align*}
		\lambda_{0}^{S_1,S_1}=&2\left(\lambda_{0}||\psi_{0}^{S_1}||^2_{L^2(\Omega_0)}-\hat{h}_{\alpha}[\psi^{S_1}_{0}]\right)+\frac{3}{S^2}||\nabla\psi_{0}||^2_{L^2(\Omega_0)}+\frac{5\alpha}{4S^2}||\psi_{0}||^2_{L^2(\partial\Omega_{0})}
	\end{align*}
	\item \begin{align*}
		\lambda_0^{a_1,a_2}=&-2\frac{c_0}{S}\left(\langle\partial_1\psi_{0},\partial_2\psi_{0}^{a_2}\rangle_{L^2(\Omega_0^+)} +\langle\partial_1\psi_0^{a_2},\partial_2\psi_{0}\rangle_{L^2(\Omega_0^+)}\right)\\
		+&2\frac{\alpha c_0}{|\Gamma_{0}|^2}\left(\langle\psi_{0},\psi_{0}^{a_2}\rangle_{L^2(\Gamma_{0}^{(1,1)})}-\langle\psi_{0},\psi_{0}^{a_2}\rangle_{L^2(\Gamma_{0}^{(2,1)})}\right)
	\end{align*}                                               \item 
	\begin{align*}
		\lambda_{0}^{a_j,c}=0
	\end{align*}
	\item 
	\begin{align*}
		\lambda_{0}^{S_1,a_j}=0
	\end{align*}
	\item 
	\begin{align*}
		\lambda_{0}^{S_1,c}=0
	\end{align*}
\end{enumerate}
\end{lem}
\begin{proof}
(I-III) Apply Lemma \ref{eigprop} alongside the computations found in Lemma \ref{explicddlam} in the appendix.\\
(IV) Apply equation (\ref{lamdiv2}) alongside the computations found in Lemma \ref{explicddlam} in the appendix.\\
(V)  We have by  (\ref{lamdiv2}) and Lemma \ref{explicddlam}
\begin{align*}
	{\lambda}^{a_j,c}_{0}=&2\hat{h}_{\alpha}^{d_{a_j}}[\psi_{0},\psi_{0}^{c}]+\hat{h}_{\alpha}^{d_{a_j},d_{c}}[\psi_{0}],\\
	=&(-1)^j\frac{2c_0}{S}\left(\langle\partial_1\psi_{0},\partial_2\psi_{0}^c\rangle_{L^2(\Omega_0\cap\lbrace (-1)^jy<0\rbrace)}+\langle\partial_1\psi_0^c,\partial_2\psi_{0}\rangle_{L^2(\Omega_0\cap\lbrace (-1)^jy<0\rbrace)}  \right)\\
	&+\frac{2\alpha c_0}{|\Gamma_{0}|^2}\left(\langle\psi_{0},\psi_0^c\rangle_{L^2(\Gamma_{0}^{(1,j)})}-\langle\psi_{0},\psi_0^c\rangle_{L^2(\Gamma_{0}^{(2,j)})}\right)\\
	=&0,
\end{align*}
where we have made use of the observation that $\psi_0^c$ is even in $x.$\\
(VI) We have again by (\ref{lamdiv2}) and Lemma \ref{explicddlam}
\begin{align*}
	{\lambda}^{a_j,S_1}_{0}=&2\hat{h}_{\alpha}^{d_{a_j}}[\psi_{0},\psi_{0}^{S_1}]+\hat{h}_{\alpha}^{d_{a_j},d_{S_1}}[\psi_{0}],\\
	=&(-1)^j\frac{2c_0}{S}\left(\langle\partial_1\psi_{0},\partial_2\psi_{0}^{S_1}\rangle_{L^2(\Omega_0\cap\lbrace (-1)^jy<0\rbrace)}+\langle\partial_1\psi_0^{S_1},\partial_2\psi_{0}\rangle_{L^2(\Omega_0\cap\lbrace (-1)^jy<0\rbrace)}  \right)\\
	&+\frac{2\alpha c_0}{|\Gamma_{0}|^2}\left(\langle\psi_{0},\psi_0^{S_1}\rangle_{L^2(\Gamma_{0}^{(1,j)})}-\langle\psi_{0},\psi_0^{S_1}\rangle_{L^2(\Gamma_{0}^{(2,j)})}\right)\\
	=&0,
\end{align*}
where we have used the fact that $\psi_{0}^{S_1}$ is even in $x.$\\
(VII) Using (\ref{lamdiv2}) and Lemma \ref{explicddlam} we have
\begin{align*}
	\lambda_{0}^{S_1,c}=&\frac{4}{S}\left(\langle\partial_2\psi_{0},\partial_2\psi_0^{c}\rangle_{L^2(\Omega_{0}^-)}-\langle\partial_2\psi_{0},\partial_2\psi_0^{c}\rangle_{L^2(\Omega_{0}^+)}\right)\\
	&+2\frac{\alpha S}{|\Gamma_{0}|}\left(\frac{1}{|\Gamma_{0}|c_0^2}-\frac{|\Gamma_{0}|}{S^2}\right)\left(\langle\psi_{0},\psi_0^{c}\rangle_{L^2(\Gamma_{0}^{(1,1)})}+\langle\psi_{0},\psi_0^{c}\rangle_{L^2(\Gamma_{0}^{(2,1)})}\right.\\
	&\phantom{=}\left.-\langle\psi_{0},\psi_0^{c}\rangle_{L^2(\Gamma_{0}^{(1,2)})}-\langle\psi_{0},\psi_0^{c}\rangle_{L^2(\Gamma_{0}^{(2,2)})}\right)\\
	=&0,
\end{align*}
since $\psi_0^c$ is even in $y.$
\end{proof}

\begin{proof}[Proof of Theorem 1]

Recall that $\lambda_{0}$ is a local maximum if it is a critical point and the Hessian matrix is negative definite.\\
The Hessian matrix for $\lambda_{0}$ is 
\begin{align*}
	H_{\lambda_0}=\left[\begin{matrix}
		\lambda_0^{a_1,a_1} & \lambda_0^{a_1,a_2} & \lambda_0^{a_1,c} &
		\lambda_0^{a_1,S_1} \\
		\lambda_0^{a_2,a_1} & \lambda_0^{a_2,a_2} & \lambda_0^{a_2,c} &
		\lambda_0^{a_2,S_1}\\
		\lambda_0^{c,a_1} & \lambda_0^{c,a_2} & \lambda_0^{c,c} &
		\lambda_0^{c,S_1} \\
		\lambda_0^{S_1c,a_1} & \lambda_0^{S_1,a_2} & \lambda_0^{S_1,c} &
		\lambda_0^{S_1,S_1}
	\end{matrix}\right].
\end{align*}
From Lemma \ref{ddexplicit} we have for $j=1,2,$  $\lambda_0^{c,a_j}=\lambda_{0}^{a_j,S_1}=\lambda_{0}^{c,S_1}=0,$  and thus $H_{\lambda_0}$ takes the form of a block diagonal matrix
\begin{align*}
	H_{\lambda_0}=\left[\begin{matrix}
		\lambda_0^{a_1,a_1} & \lambda_0^{a_1,a_2} & 0 &
		0\\
		\lambda_0^{a_2,a_1} & \lambda_0^{a_2,a_2} & 0 &
		0\\
		0 & 0 & \lambda_0^{c,c} &
		0 \\
		0 & 0 & 0 &
		\lambda_0^{S_1,S_1}
	\end{matrix}\right],
\end{align*}                                                  and hence its eigenvalues $\mu_i$ are the eigenvalues of the diagonal blocks.  
	Hence, $\mu_1=\lambda_{0}^{c,c}$ and $\mu_2=\lambda_{0}^{S_1,S_1}.$\\
	By Lemma  \ref{eigprop} we see that $\lambda_{0}^{v,v}<0$ if $\hat{h}_\alpha^{d_v,d_v}[\psi_{0}]<0.$ But by comparing Lemmas \ref{eigprop} and \ref{ddexplicit} we see that for $v=a_1,a_2,c,S_1,$  we have  $\hat{h}_\alpha^{d_v,d_v}[\psi_{0}]=B(||\nabla\psi_{0}||^2_{L_2(\Omega_{0})}+\alpha C||\psi_{0}||^2_{L^2(\partial\Omega_{0})}),$ for some $0<B,$ $0<C<1$ depending upon $v.$ By Lemma \ref{extre} we have        $||\nabla\psi_{0}||^2_{L_2(\Omega_{0})}+\alpha C||\psi_{0}||^2_{L^2(\partial\Omega_{0})}<0,$ and so   $\lambda_{0}^{v,v}<0.$\\
	This implies that we have $\mu_1,\mu_2<0.$ \\
	The other eigenvalues are the eigenvalues of the upper left block matrix, whose trace and determinant give the following for the sum and product of eigenvalues
	\begin{align}
		\mu_3+\mu_4=&\lambda_{0}^{a_1,a_1}+\lambda_{0}^{a_2,a_2},\label{eq1}\\
		\mu_3\mu_4=&\lambda_{0}^{a_1,a_1}\lambda_{0}^{a_2,a_2}-\left(\lambda_{0}^{a_1,a_2}\right)^2.\nonumber
	\end{align}
	Since $\lambda_{0}^{a_j,a_j}<0,$ for $j=1,2$ (\ref{eq1}) implies at least one of $\mu_3$ or $\mu_4$ is strictly negative. Thus for both $\mu_3<0$ and $\mu_4<0$ we need $\mu_3\mu_4>0.$ But
	\begin{align*}
		\lambda_{0}^{a_1,a_1}\lambda_{0}^{a_2,a_2}-\left(\lambda_{0}^{a_1,a_2}\right)^2=&\left(\frac{1}{2S}||\nabla\psi_{0}||^2_{L^2(\Omega_0)}+\alpha\frac{1}{8S}||\psi_{0}||^2_{L^2(\partial\Omega_0)}\right)^2\\
		-&2\left(\frac{1}{2S}||\nabla\psi_{0}||^2_{L^2(\Omega_0)}+\alpha\frac{1}{8S}||\psi_{0}||^2_{L^2(\partial\Omega_0)}\right)\left(\mathscr{G}[\psi_0^{a_1},\psi_0^{a_1}]+\mathscr{G}[\psi_0^{a_2},\psi_0^{a_2}]\right)\\   
		+&        4(\mathscr{G}[\psi_0^{a_1},\psi_0^{a_1}]\mathscr{G}[\psi_0^{a_2},\psi_0^{a_2}]          -\mathscr{G}[\psi_0^{a_1},\psi_0^{a_2}]^2),
	\end{align*}
	where $\mathscr{G}$ is given by $\mathscr{G}[f,g]=\hat{h}_\alpha[f,g]-\lambda_{0}\langle f,g\rangle_{L^2(\Omega_{0})}.$
	Clearly the first two terms are strictly positive and the last term can be seen to be non-negative once one observes that $\mathscr{G}$ is a positive semi-definite Hermitian form, and hence the Cauchy-Schwarz inequality holds for  $\mathscr{G}.$\\
	Since $\lambda_{0}$ is a local maximum in the parameter space $(a_1,a_2,c,S_1)$ it is also a local maximum with respect to the Hausdorf metric. \end{proof}
\section{Quadrilateral Global Results}
This section collects various asymptotic results for small and large negative $\alpha,$ provable via test function arguments. The approach is similar to that of \cite{krejvcivrik2023reverse}, but adapted to quadrilaterals. In particular these results will prove Theorems \ref{thm2} and \ref{thm3}. \\
Again, throughout this section we make the normalisation $||\psi_0||^2_{L^2(\Omega_0)}=1.$                                                                                                                                                                                                                                                 \begin{prop}\label{smalph}
	Let $\Omega_{a_1,a_2,c,S_1}$ be a quadrilateral other than the square. Then there exists a negative constant $\alpha_c,$ depending only on $a_1,a_2,c,S_1$ and $S$ such that $\lambda_{a_1,a_2,c,S_1}<\lambda_{0}$ for all $\alpha$ satisfying $\alpha_c\leq\alpha<0.$
\end{prop}
\begin{proof}
	First note that
	\begin{align*}
		\lambda_{a_1,a_2,c,S_1}\leq& {\hat{h}_{\alpha,a_1,a_2,c,S_1}[\psi_{0}]}\\
		\implies \lambda_{a_1,a_2,c,S_1}\leq & \lambda_{0}+ {\hat{h}_{\alpha,a_1,a_2,c,S_1}[\psi_{0}]-{h}_{\alpha}[\psi_{0}]},
	\end{align*}
	and so ${\hat{h}_{\alpha,a_1,a_2,c,S_1}[\psi_{0}]-{h}_{\alpha}[\psi_{0}]}<0$ is a sufficient condition for $\lambda_{a_1,a_2,c,S_1}< \lambda_{0}.$
	
	Next observe using the symmetry properties from Corollary \ref{symprop} that 
	\begin{align*}
		&\left|\left|\left(\bchi_{y<0}\frac{a_2}{S_2}-\bchi_{y>0}\frac{a_1}{S_1}\right)c_0\partial_1\psi_0+\frac{cS}{c_0\left(S_1\bchi_{y>0}+S_2\bchi_{y<0}\right)}\partial_2\psi_0\right|\right|^2_{L^2(\Omega_0)}\\
		&=\frac{a_1^2c_0^2}{S_1^2}||\partial_1\psi_{0}||_{L^2(\Omega_{0}^+)}^2-\frac{2a_1cS}{S_1^2}\langle\partial_1\psi_0,\partial_2\psi_0\rangle_{L^2(\Omega_{0}^+)}+\frac{c^2S^2}{c_0^2S_1^2}||\partial_2\psi_{0}||_{L^2(\Omega_{0}^+)}^2\\
		&\phantom{a}+\frac{a_2^2c_0^2}{S_2^2}||\partial_1\psi_{0}||_{L^2(\Omega_{0}^-)}^2+\frac{2a_2cS}{S_2^2}\langle\partial_1\psi_0,\partial_2\psi_0\rangle_{L^2(\Omega_{0}^-)}+\frac{c^2S^2}{c_0^2S_2^2}||\partial_2\psi_{0}||_{L^2(\Omega_{0}^-)}^2\\                                                   &=\left(\frac{a_1^2c_0^2}{4S_1^2}+\frac{a_2^2c_0^2}{4S_2^2}+\frac{c^2S^2}{4c_0^2S_1^2}+\frac{c^2S^2}{4c_0^2S_2^2}\right)||\nabla\psi_{0}||_{L^2(\Omega_{0})}^2,
	\end{align*}
	and so we may write equation (\ref{explexpl}) in the form
	\begin{align*}
		\hat{h}_{\alpha,a_1,a_2,c,S_1}[\psi_0]=&
		\left(\frac{a_1^2c_0^2}{4S_1^2}+\frac{a_2^2c_0^2}{4S_2^2}+\frac{c^2S^2}{4c_0^2S_1^2}+\frac{c^2S^2}{4c_0^2S_2^2}+\frac{c_0^2}{2c^2}\right)||\nabla\psi_{0}||_{L^2(\Omega_{0})}^2\\
		+&\frac{\alpha S}{4\sqrt{2}c_0}l(a_1,a_2,c,S_1)||\psi_{0}||^2_{L^2(\partial\Omega_{0})},
	\end{align*}
	where
	\begin{align*}
		&l(a_1,a_2,c,S_1)=\sum_{i,j}\frac{|\Gamma_{a_j,c,S_1}^{(i,j)}|}{S_j}\\=&\frac{\sqrt{\frac{S_1^2}{c^2}+(a_1+c)^2}+\sqrt{\frac{S_1^2}{c^2}+(a_1-c)^2}}{S_1}
		+\frac{\sqrt{\frac{S_2^2}{c^2}+(a_2+c)^2}+\sqrt{\frac{S_2^2}{c^2}+(a_2-c)^2}}{S_2}.
	\end{align*}
	Via applying straightforward  optimisation arguments in each variable one finds that
	\begin{align}
		l(a_1,a_2,c,S_1)\geq& 2\frac{\sqrt{\frac{S_1^2}{c^2}+c^2}}{S_1}+2\frac{\sqrt{\frac{S_2^2}{c^2}+c^2}}{S_2}\label{lineq}\\
		\geq& \frac{2\sqrt{2}}{\sqrt S_1}+\frac{2\sqrt{2}}{\sqrt S_2}\label{l2}\\
		\geq& \frac{4\sqrt{2S}}{S},\label{l3}
	\end{align}
	where we have equality only in the case $a_1=a_2=0,$ $c=c_0,$ $S_1=S.$\\                                        
	Since we have 
	\begin{align*}
		\frac{S^2}{S_1^2}+\frac{S^2}{S_2^2}\geq2,
	\end{align*}                                 
	we also have
	\\                
	\begin{align}\label{auxb}
		\frac{c^2S^2}{4c_0^2S_1^2}+\frac{c^2S^2}{4c_0^2S_2^2}+\frac{c_0^2}{2c^2}  	\geq\frac{c^2}{2c_0^2}+\frac{c_0^2}{2c^2}
		=\frac{(c^2-c_0^2)^2+2c^2c_0^2}{2c^2c_0^2}
		\geq1.
	\end{align}                                          
	Let 
	\begin{align*}
		z(a_1,a_2,c,S_1)=&\frac{\frac{Sl(a_1,a_2,c,S_1)}{4\sqrt{2}c_0}-1}{\frac{a_1^2c_0^2}{4S_1^2}+\frac{a_2^2c_0^2}{4S_2^2}+\frac{c^2S^2}{4c_0^2S_1^2}+\frac{c^2S^2}{4c_0^2S_2^2}+\frac{c_0^2}{2c^2}-1},\\
		g(\alpha)=&-\frac{||\nabla\psi_{0}||_{L^2(\Omega_{0})}^2}{\alpha||\psi_{0}||^2_{L^2(\partial\Omega_{0})}},
	\end{align*}
	observing from (\ref{l3}) and (\ref{auxb}) that $z(a_1,a_2,c,S_1)\geq 0$ with equality if and only if $(a_1,a_2,c,S_1)=(0,0,c_0,S).$\\
	Now, note that
	\begin{align*}
		\hat{h}_{\alpha,a_1,a_2,c,S_1}[\psi_{0}]-\hat{h}_\alpha[\psi_0]<& 0\\
		\iff 0> & 
		\left(\frac{a_1^2c_0^2}{4S_1^2}+\frac{a_2^2c_0^2}{4S_2^2}+\frac{c^2S^2}{4c_0^2S_1^2}+\frac{c^2S^2}{4c_0^2S_2^2}+\frac{c_0^2}{2c^2}-1\right)||\nabla\psi_{0}||_{L^2(\Omega_{0})}^2\\
		&+\alpha\left(\frac{Sl(a_1,a_2,c,S_1)}{4\sqrt{2}c_0}-1\right)||\psi_{0}||^2_{L^2(\partial\Omega_{0})}\\
		\iff 0< &\frac{||\nabla\psi_{0}||_{L^2(\Omega_{0})}^2}{\alpha||\psi_{0}||^2_{L^2(\partial\Omega_{0})}}+\frac{\frac{Sl(a_1,a_2,c,S_1)}{4\sqrt{2}c_0}-1}{\frac{a_1^2c_0^2}{4S_1^2}+\frac{a_2^2c_0^2}{4S_2^2}+\frac{c^2S^2}{4c_0^2S_1^2}+\frac{c^2S^2}{4c_0^2S_2^2}+\frac{c_0^2}{2c^2}-1}\\
		\iff g(\alpha&)< z(a_1,a_2,c,S_1).
	\end{align*}
	Recall that 
	\begin{align*}
		\lambda_{0}=||\nabla\psi_0||^2_{L^2(\Omega_0)}+\alpha||\psi_0||^2_{L^2(\partial\Omega_0)},
	\end{align*}
	with $\lambda_0$ converging to the first Neumann eigenvalue in the limit $\alpha\to 0.$ Further,
	\begin{align*}
		\frac{d\lambda_{0}}{d\alpha}=||\psi_0(\alpha)||^2_{L^2(\partial\Omega_0)},
	\end{align*}
	(see for instance \cite[p.~86-89]{henrot2017shape}.)\\                                                                                                                  Recall, for the Neumann problem that the first eigenfunction is the constant $\psi_0=|\Omega_0|^{-1/2}.$\\
	At $\alpha=0$ we have
	\begin{align*}
		\left.\frac{d\lambda_{0}}{d\alpha}\right|_{\alpha=0}=||\psi_0(0)||^2_{L^2(\partial\Omega_0)}=\frac{|\partial\Omega_0|}{|\Omega_0|}.
	\end{align*}
	Further, by the definition of derivative, we have
	\begin{align}
		\frac{|\partial\Omega_0|}{|\Omega_0|}=\left.\frac{d\lambda_{0}}{d\alpha}\right|_{\alpha=0}=\lim\limits_{\alpha\to 0}\frac{\lambda_{0}(\alpha)-\lambda_0(0)}{\alpha}=\lim\limits_{\alpha\to 0}\frac{||\nabla\psi_0||^2_{L^2(\Omega_0)}+\alpha||\psi_0||^2_{L^2(\partial\Omega_0)}}{\alpha}.\label{eq:eigarg}
	\end{align}
	Now,\\
	\begin{align*}
		\lim\limits_{\alpha\to 0}g(\alpha)=&-\lim\limits_{\alpha\to 0}\frac{||\nabla\psi_0||^2_{L^2(\Omega_0)}}{\alpha||\psi_0||^2_{L^2(\partial\Omega_0)}}\\
		=&1-\lim\limits_{\alpha\to 0}\frac{||\nabla\psi_0||^2_{L^2(\Omega_0)}+\alpha||\psi_0||^2_{L^2(\partial\Omega_0)}}{\alpha||\psi_0||^2_{L^2(\partial\Omega_0)}}\\=&1-\lim\limits_{\alpha\to 0}\left(\frac{1}{||\psi_0||^2_{L^2(\partial\Omega_0)}}\right)\lim\limits_{\alpha\to 0}\left(\frac{||\nabla\psi_0||^2_{L^2(\Omega_0)}+\alpha||\psi_0||^2_{L^2(\partial\Omega_0)}}{\alpha}\right),\\
		=&1-\frac{|\Omega_0|}{|\partial\Omega_0|}\frac{|\partial\Omega_0|}{|\Omega_0|},\\
		=&0,
	\end{align*}
	where we have used (\ref{eq:eigarg}).\\                                                                                                                                                                                                                                                                      Now, since $g(\alpha)> 0$ for $\alpha<0$ and $g(\alpha)\to0$ as $\alpha\to 0,$ it follows that since $z(a_1,a_2,c,S_1)>0$ there exists $\alpha_c$ such that $0>\alpha>\alpha_c$ implies $0<g(\alpha)<z(a_1,a_2,c,S_1).$ 
\end{proof}
The asymptotic behaviour of the first Robin eigenvalue  with large negative $\alpha$ is known for piecewise smooth domains, see for instance the work of 
Levitin and Parnovski \cite{levitin2008principal}. In particular we will make use of the following result which follows from Theorem 3.6 in \cite{levitin2008principal},
\begin{lem}
	Let $\Omega_{a_1,a_2,c,S_1}$ be a quadrilateral with inner angles $\theta_1,\theta_2,\theta_3,\theta_4$ respectively. Let \begin{align*}
		C_{i}=\begin{cases}
			1, \quad &\text{if }\quad\theta_i\geq\pi;\\
			\csc^2\left(\theta_i/2\right),\quad &\text{if }\quad\theta_i\leq\pi.
		\end{cases}
	\end{align*}
	Then,
	\begin{align}\label{eqn:draft}
		\lambda_{a_1,a_2,c,S_1}=-\alpha^2\max\limits_{i\in\lbrace1,2,3,4\rbrace}C_{i}+o(\alpha^2),\quad \alpha\to-\infty.
	\end{align}
\end{lem} 
\begin{cry}\label{sectcry2}
	Let $\Omega_{a_1,a_2,c,S_1}$ be a quadrilateral other than a rectangle. Then, there exists $\alpha_c<0$ such that for all $\alpha<\alpha_c,$ 
	\begin{align*}
		\lambda_{a_1,a_2,c,S_1}<\lambda_{0}.
	\end{align*} 
\end{cry}
\begin{proof}
	Observe from Lemma (\ref{eqn:draft}) above that
	\begin{align*}
		\lambda_{0}=-2\alpha^2+o(\alpha^2),\quad \alpha\to-\infty,
	\end{align*}
	and thus it is clear that the corollary follows if 
	\begin{align}
		-\alpha^2\max\limits_{i\in\lbrace1,2,3,4\rbrace}C_{i}<-2\alpha^2,\label{eq:ineq}
	\end{align}
	with $\theta_i$ and $q_i$ given as in the previous lemma. However, since $\Omega_{a_1,a_2,c,S_1}$ is not a rectangle, at least one of its inner angles, $\theta_1$ say, satisfies $\theta_1<\pi/2,$ and so $\max\limits_{i\in\lbrace1,2,3,4\rbrace}C_{q_i}>\csc^2(\theta_1/2)>2$ showing (\ref{eq:ineq}).
\end{proof}
\begin{remark}
	Note that although the above Corollary does not hold for rectangles, it is already known via a separation of variables argument that the square is the maximiser of the first eigenvalue among rectangular domains, see for instance \cite{laugesen2019robin}.
\end{remark} 
\begin{proof}[Proof of Theorem 2]
	Proposition \ref{smalph}, Corollary \ref{sectcry2} and the above Remark together prove Theorem \ref{thm2}.\qed
\end{proof}  
We can also use test function arguments to obtain asymptotic results for extreme values of the quadrilateral parameters.
\begin{prop}\label{quadasyp}
	Let $\alpha<0,$ then there exist positive constants $A,c_1,c_2,\tilde{S}$ depending only upon $\alpha$ such that any of the following restrictions\begin{enumerate}[label=(\Roman*)]
		\item $|a_1|>A, \ a_2\in\mathbb{R}, \ \ c>0 \text{ and } 0<S_1<2S;$
		\item $|a_2|>A, \ a_1\in\mathbb{R}, \ \ c>0 \text{ and } 0<S_1<2S;$
		\item $a_1\in\mathbb{R},\ a_2\in\mathbb{R}, \ \ 0<S_1<2S \text{ and } c>c_1;$
		\item $a_1\in\mathbb{R},\ a_2\in\mathbb{R}, \ \ 0<S_1<2S \text{ and }  c<c_2;$
		\item $a_1\in\mathbb{R},\ a_2\in\mathbb{R}, \ \ S_1<\tilde{S} \text{ and }  c>0;$     
		\item $a_1\in\mathbb{R},\ a_2\in\mathbb{R}, \ \ |2S-S_1|<\tilde{S} \text{ and }  c>0;$                                                                                                                                                                                                                                     \end{enumerate}
	imply that $\lambda_{a_1,a_2,c,S_1}<\lambda_{0}.$
\end{prop}
\begin{proof}
	Consider the trial function $ \mathds{1} \in L^2(\Omega_{0}).$\\
	Then,
	\begin{align*}
		\hat{h}_{\alpha,a_1,a_2,c,S_1}[\mathds{1}]=&\alpha\sum_{i,j}\frac{S|\Gamma_{a_j,c,S_1}^{(i,j)}|}{S_j|\Gamma_{0}^{(i)}|}||\mathds{1}||^2_{L^2(\Gamma^{(i)}_{0})}\\=&\alpha\sum_{i,j}\frac{S|\Gamma_{a_j,c,S_1}^{(i,j)}|}{S_j},
	\end{align*}
	and hence,
	\begin{align*}
		\lambda_{a_1,a_2,c,S_1}\leq\frac{\hat{h}_{\alpha,a_1,a_2,c,S_1}[\mathds{1}]}{||\mathds{1}||^2_{L^2(\Omega_0)}}=\frac{\alpha}{2}\sum_{i,j}\frac{|\Gamma_{a_j,c,S_1}^{(i,j)}|}{S_j}.
	\end{align*}
	As in the proof of Proposition (\ref{smalph})                                            we set $l(a_1,a_2,c,S_1)=\sum_{i,j}\frac{|\Gamma_{a_j,c,S_1}^{(i,j)}|}{S_j},$ and then observe that we have
	\begin{align*}
		\lambda_{a_1,a_2,c,S_1}<\lambda_{0},
	\end{align*}
	provided that the sufficent condition
	\begin{align*}
		l(a_1,a_2,c,S_1)>\frac{2\lambda_{0}}{\alpha}
	\end{align*}
	is satisfied.\\
	Observe that for $j=1,2,$
	\begin{align*}
		l(a_1,a_2,c,S_1)\geq& \frac{\sqrt{\frac{S_j^2}{c^2}+(a_j+c)^2}+\sqrt{\frac{S_j^2}{c^2}+(a_j-c)^2}}{S_j}\\
		\geq&\frac{\sqrt{\frac{S_j^2}{c^2}+(|a_j|+c)^2}}{S_j}\\
		\geq&\frac{|a_j|}{2S}\to\infty \text{ as }|a_j|\to\infty.
	\end{align*}                                                       This proves (I) and (II).\\                                                                                                                                                                                                                                                                                                                                                                                           Recall from equation (\ref{lineq}) in the proof of Proposition \ref{smalph} that
	\begin{align*}
		l(a_1,a_2,c,S_1)\geq2\frac{\sqrt{\frac{S_1^2}{c^2}+c^2}}{S_1}+2\frac{\sqrt{\frac{S_2^2}{c^2}+c^2}}{S_2},
	\end{align*}
	and so
	\begin{align*}
		l(a_1,a_2,c,S_1)\geq&2{\sqrt{\frac{1}{c^2}+\frac{c^2}{S_1^2}}}\\
		\geq&2{\sqrt{\frac{1}{c^2}+\frac{c^2}{4S^2}}}.
	\end{align*}
	Observing that $\lim\limits_{c\to\infty} {\sqrt{\frac{1}{c^2}+\frac{c^2}{4S^2}}}=\infty$ and          $\lim\limits_{c\to 0} {\sqrt{\frac{1}{c^2}+\frac{c^2}{4S^2}}}=\infty$ proves (III) and (IV).                                                                                                                                                                                          \\
	Next, recall from equation (\ref{l2}) in the proof of Proposition \ref{smalph} that   \begin{align*}
		l(a_1,a_2,c,S_1)\geq \frac{2\sqrt{2}}{\sqrt{S_1}}+\frac{2\sqrt{2}}{\sqrt{S_2}},
	\end{align*}   
	where $\lim\limits_{S_1\to 0}   \frac{2\sqrt{2}}{\sqrt{S_1}}+\frac{2\sqrt{2}}{\sqrt{S_2}}=\infty$ and    $\lim\limits_{S_1\to 2S}   \frac{2\sqrt{2}}{\sqrt{S_1}}+\frac{2\sqrt{2}}{\sqrt{S_2}}=\infty$  which proves (V) and (VI).      
\end{proof}  
\begin{proof}[Proof of Theorem 3]
Given $A,c_1,c_2,\tilde{S}$ from Proposition \ref{quadasyp}, we can find $C$ such that $d_H(\Omega,\Omega_0)>C$ implies that the conditions ($I$)-($VI$) hold.
\end{proof}                     
\section{Appendix}
This appendix consists primarily of messy computational details omitted in the preceding sections. In particular, we provide the details needed in the computation of the first and second derivatives of the first eigenvalue with respect to the geometric parameters.\\
\begin{lem}\label{hd1v2}
	Explicitly, we have
	\item
\begin{enumerate}[label=(\Roman*)]
	\item for $j=1,2,$
	\begin{align*}
		\hat{h}_{\alpha,a_1,a_2,c,S_1}^{d_{a_j}}[f,\phi]=&\left\langle(-1)^j\frac{c_0}{S}\partial_1f,(-1)^j\frac{a_jc_0}{S}\partial_1\phi+\frac{cS}{c_0S_j}\partial_2\phi\right\rangle_{L^2(\Omega_{0}\cap\lbrace (-1)^jy<0\rbrace)}\\
		+&\left\langle(-1)^j\frac{c_0}{S}\partial_1\phi,(-1)^j\frac{a_jc_0}{S}\partial_1f+\frac{cS}{c_0S_j}\partial_2f\right\rangle_{L^2(\Omega_{0}\cap\lbrace(-1)^jy<0\rbrace)}\\
		+&\frac{\alpha S}{S_j|\Gamma_0|}\sum_{i=1}^{2}\frac{a_j+(-1)^{i+1}c}{|\Gamma^{(i,j)}_{a_j,c,S_1}|}\langle f,\phi\rangle_{L^2(\Gamma_{0}^{(i,j)})},
	\end{align*}
	\item 
	\begin{align*}
		\hat{h}_{\alpha,a_1,a_2,c,S_1}^{d_{c}}[f,\phi]=&-2\frac{c_0^2}{c^3}\langle\partial_1f,\partial_1\phi\rangle_{L^2(\Omega_0)}\\
		+&\sum_{j=1}^{2}\left\langle\frac{S}{S_jc_0}\partial_2f,(-1)^{j+1}a_j\frac{c_0}{S_j}\partial_1\phi+\frac{cS}{c_0S_j}\partial_2\phi\right\rangle_{L^2(\Omega_0\cap{\lbrace (-1)^jy<0\rbrace})}\\
		+&\sum_{j=1}^{2}\left\langle\frac{S}{S_jc_0}\partial_2\phi,(-1)^{j+1}a_j\frac{c_0}{S_j}\partial_1f+\frac{cS}{c_0S_j}\partial_2f\right\rangle_{L^2(\Omega_0\cap{\lbrace (-1)^jy<0\rbrace})}\\
		+&\sum_{i,j}\alpha\frac{S\left(-\frac{S^2}{c^3}+(-1)^{i+1}a_j+c\right)}{S_j|\Gamma_{0}||\Gamma^{(i,j)}_{a_j,c,S_1}|}\langle f,\phi\rangle_{L^2(\Gamma_{0}^{(1,1)})},
	\end{align*}
	\item \begin{align*}
		&\hat{h}_{\alpha,a_1,a_2,c,S_1}^{d_{S_1}}[f,\phi]=\frac{\alpha S}{|\Gamma_{0}|}\sum_{i,j}(-1)^{j+1}\left(\frac{1}{|\Gamma^{(i,j)}_{a_j,c,S_1}|c^2}-\frac{|\Gamma^{(i,j)}_{a_j,c,S_1}|}{S_j^2}\right)\langle f,\phi\rangle_{L^2(\Gamma_{0}^{(i,j)})}\\
		+&\sum_{j=1}^2\left\langle\frac{a_jc_0}{S_j^2}\partial_1f+(-1)^j\frac{cS}{S_j^2c_0}\partial_2f,(-1)^j\frac{a_jc_0}{S_j}\partial_1\phi+\frac{cS}{c_0S_j}\partial_2\phi\right\rangle_{L^2(\Omega_{0}\cap \lbrace (-1)^jy<0\rbrace)},\\
		+&\sum_{j=1}^{2}\left\langle\frac{a_jc_0}{S_j^2}\partial_1\phi+(-1)^j\frac{cS}{S_j^2c_0}\partial_2\phi,(-1)^j\frac{a_jc_0}{S_j}\partial_1f+\frac{cS}{c_0S_j}\partial_2f\right\rangle_{L^2(\Omega_{0}\cap \lbrace (-1)^jy<0\rbrace)}.
	\end{align*}
	\end{enumerate}                                                                                                                                                                                                                                     \end{lem}
\begin{proof}
	Via direct computation using the formulae (\ref{hdif}).
\end{proof} 
The following lemma contains the computations of the second order terms that are required for the computation of  ${\lambda}^{v_1,v_2}_{a_1,a_2,c,S_1}$ via (\ref{lamdiv2}).         \begin{lem}      \label{2nddif}                                                                                                        
	Explicitly we have \begin{enumerate}[label=(\Roman*)]
			\item for $j=1,2,$
			\begin{align*}
				\hat{h}_{\alpha,a_1,a_2,c,S_1}^{d_{a_j},d_{a_j}}[f]=&2\frac{c_0^2}{S_1^2}||\partial_1f||^2_{L^2(\Omega_0\cap\lbrace (-1)^jy<0\rbrace)}\\
				&+\frac{\alpha S}{S_j|\Gamma_0|}\sum_{i=1}^{2}\left(\frac{1}{|\Gamma^{(i,j)}_{a_j,c,S_1}|}-\frac{(a_1+(-1)^{i+1}c)^2}{|\Gamma^{(i,j)}_{a_j,c,S_1}|^3}\right)||f||^2_{L^2(\Gamma_0^{(i,j)})},
			\end{align*}        
			\item 
			\begin{align*}
				\hat{h}_{\alpha,a_1,a_2,c}^{d_{c},d_{c}}[f]=&6\frac{c_0^2}{c^4}||\partial_1f||^2_{L^2(\Omega_{0})}+\frac{2}{c_0^2}\sum_{j=1}^2\frac{S^2}{S_j^2}\left|\left|\partial_2f\right|\right|^2_{L^2(\Omega_{0}\cap\lbrace(-1)^jy<0\rbrace)}\\
				+&\frac{\alpha S}{|\Gamma_{0}|}\sum_{i,j}\frac{1}{S_j}\left(\frac{3\frac{S_j^2}{c^4}+1}{|\Gamma_{a_j,c,S_1}^{(i,j)}|}-\frac{\left(-\frac{S_j^2}{c^3}+(-1)^{i+1}a_j+c\right)^2}{|\Gamma_{a_j,c,S_1}^{(i,j)}|^3}\right)||f||^2_{L^2(\Gamma_{0}^{(1,1)})},
			\end{align*}       
			\item
			\begin{align*}
				&\hat{h}_{\alpha,a_1,a_2,c}^{d_{S_1},d_{S_1}}[f]=2\sum_{j=1}^{2}\left|\left|\frac{a_jc_0}{S_j^2}\partial_1f+(-1)^j\frac{cS}{c_0S_j^2}\partial_2f\right|\right|^2_{L^2(\Omega_{0}\cap\lbrace (-1)^jy<0\rbrace)}\\+&2\sum_{j=1}^{2}\left\langle(-1)^j\frac{a_j c_0}{S_j}\partial_1f+\frac{cS}{c_0S_j}\partial_2f,(-1)^j\frac{2a_jc_0}{S_j^3}\partial_1f+\frac{2cS}{c_0S_j^3}\partial_2f\right\rangle_{L^2(\Omega_{0}\cap\lbrace (-1)^jy<0\rbrace)}\\
				&+\frac{\alpha S}{|\Gamma_{0}|}\sum_{i,j}\left(-\frac{S_j}{|\Gamma^{(i,j)}_{a_j,c,S_1}|^3c^4}-\frac{1}{c^2|\Gamma^{(i,j)}_{a_j,c,S_1}|S_j}+\frac{2|\Gamma^{(i,j)}_{a_1,c,S_1}|}{S_j^3}\right)||f||^2_{L^2(\Gamma_{0}^{(i,j)})},
			\end{align*}  
			\item for $j=1,2,$ \begin{align*}
				&\hat{h}_{\alpha,a_1,a_2,c,S_1}^{d_{a_j},d_{c}}[f]=(-1)^j\frac{2S}{S_j^2}\langle\partial_1f,\partial_2f\rangle_{L^2(\Omega_0\cap\lbrace (-1)^jy<0\rbrace)}\\
				+&\frac{\alpha S}{|\Gamma_0|S_j}\sum_{i=1}^{2}\left(\frac{(-1)^{i+1}}{|\Gamma^{(i,j)}_{a_j,c,S_1}|}-\frac{(a_j+(-1)^{i+1}c)(-\frac{S_j^2}{c^3}+(-1)^{i+1}a_j+c)}{|\Gamma^{(i,j)}_{a_j,c,S_1}|^3}\right)||f||_{\Gamma_{0}^{(i,j)}}^2,
			\end{align*}
			\item \begin{align*}
				\hat{h}_{\alpha,a_1,a_2,c}^{d_{a_1},d_{a_2}}[f]=0,                                             \end{align*} 
			\item for $j=1,2,$
			\begin{align*}
				&\hat{h}_{\alpha,a_1,a_2,c,S_1}^{d_{a_j},d_{S_1}}[f]
				=4\frac{c_0}{S_j^3}\left\langle\partial_1f,(-1)^j{a_jc_0}\partial_1f+\frac{cS}{c_0}\partial_2f\right\rangle_{L^2(\Omega_{0}\cap\lbrace (-1)^jy<0\rbrace)}\\
				&+(-1)^j\frac{\alpha S}{|\Gamma_{0}|}\sum_{i=1}^2(a_j+(-1)^{i+1}c)\left(\frac{1}{S_j^2|\Gamma^{(i,j)}_{a_j,c,S_1}|}+\frac{1}{|\Gamma_{a_j,c,S_1}^{(i,j)}|^3c^2}\right)||f||^2_{L^2(\Gamma_{0}^{(i,j)})},
			\end{align*}
			\item 
			\begin{align*}
				&\hat{h}^{d_c,{d_{S_1}}}_{\alpha,a_1,a_2,c,S_1}[f]=\frac{4S}{c_0S_j^3}\sum_{j=1}^{2}\left\langle(-1)^j\partial_2f,(-1)^j{a_jc_0}\partial_1f+\frac{cS}{c_0}\partial_2f\right\rangle_{L^2(\Omega_{0}\cap\lbrace (-1)y^j<0\rbrace)}\\
				+&\frac{\alpha S}{|\Gamma_{0}|}\sum_{i,j}(-1)^j\left(\frac{\frac{S_j^2}{c^3}+(-1)^{i+1}a_j+c}{S_j^2|\Gamma^{(i,j)}_{a_j,c,S_1}|}+\frac{-\frac{S_j^2}{c^3}+(-1)^{i+1}a_j+c}{|\Gamma_{a_j,c,S_1}^{(i,j)}|^3c^2}\right)||f||^2_{L^2(\Gamma_{0}^{(i,j)})},
			\end{align*}
		\end{enumerate}
\end{lem}                                                                  \begin{proof}
	Via direct computation, using the formulae  (\ref{dd1}-\ref{dd3}). 
\end{proof}                                            

Next, the preceding lemmas applied in the case $f=\psi_{0}$ and $(a_1,a_2,c,S_1) =(0,0,c,S)$ takes the simplified form of the following lemma.
\begin{lem}\label{explicddlam}
	Explicitly we have
	\begin{enumerate}[label=(\Roman*)]
		\item for $j=1,2,$
		\begin{enumerate}[label=(\roman*)]
			\item \begin{align*}
				\hat{h}_{\alpha}^{d_{a_j},d_{a_j}}[\psi_{0}]=\frac{1}{2S}||\nabla\psi_{0}||^2_{L^2(\Omega_0)}+\alpha\frac{1}{8S}||\psi_{0}||^2_{L^2(\partial\Omega_0)},
			\end{align*}
			\item
			\begin{align*}
				\hat{h}_{\alpha}^{d_{c},d_{c}}[\psi_{0}]=\frac{4}{S}||\nabla\psi_{0}||^2_{L^2(\Omega_0)}+\alpha\frac{2}{S}||\psi_{0}||^2_{L^2(\partial\Omega_{0})},
			\end{align*}
			\item 
			\begin{align*}
				\hat{h}_\alpha^{{d_{S_1}},{d_{S_1}}}[\psi_0]=\frac{3}{S^2}||\nabla\psi_{0}||^2_{L^2(\Omega_0)}+\frac{5\alpha}{4S^2}||\psi_{0}||^2_{L^2(\partial\Omega_{0})},
			\end{align*}
			\item for $j=1,2,$\begin{align*}
				\hat{h}_{\alpha}^{d_{a_j},d_{c}}[\psi_{0}]=0,
			\end{align*}
			\item \begin{align*}
				\hat{h}_{\alpha}^{d_{a_1},d_{a_2}}[\psi_{0}]=0,
			\end{align*}
			\item for $j=1,2$
			\begin{align*}
				\hat{h}_{\alpha}^{d_{a_j},d_{S_1}}[\psi_{0}]=0,
			\end{align*}
			\item \begin{align*}
				\hat{h}^{d_c,{d_{S_1}}}_{\alpha}[\psi_{0}]=0,
			\end{align*}
		\end{enumerate}
		\item 
		
		\begin{enumerate}
			\item for $j=1,2,$ \begin{align*}
				\hat{h}_{\alpha}^{d_{a_j}}[\psi_{0},\phi]=&(-1)^j\frac{c_0}{S}\left(\langle\partial_1\psi_{0},\partial_2\phi\rangle_{L^2(\Omega_0\cap\lbrace (-1)^jy<0\rbrace)}+\langle\partial_1\phi,\partial_2\psi_{0}\rangle_{L^2(\Omega_0\cap\lbrace (-1)^jy<0\rbrace)}\right)\\
				+&\alpha\frac{c_0}{|\Gamma_{0}|^2}\sum_{i=1}^{2}(-1)^{i+1}\langle\psi_{0},\phi\rangle_{L^2(\Gamma_{0}^{(i,j)})},
			\end{align*} 
			\item  
			\begin{align*}
				\hat{h}_{\alpha}^{d_{c}}[\psi_{0},\phi]=      \frac{2}{c_0}\left(\langle\partial_2\psi_{0},\partial_2\phi\rangle_{L^2(\Omega_0)}-\langle\partial_1\psi_{0},\partial_1\phi\rangle_{L^2(\Omega_0)}   \right)      ,               
			\end{align*}
			\item
			\begin{align*}
				\hat{h}_{\alpha}^{d_{S_1}}[\psi_{0},\phi]=&\frac{2}{S}\left(\langle\partial_2\psi_{0},\partial_2\phi\rangle_{L^2(\Omega_{0}\cap \lbrace y<0\rbrace)}-\langle\partial_2\psi_{0},\partial_2\phi\rangle_{L^2(\Omega_{0}\cap \lbrace y>0\rbrace)}\right)\\
				+&\frac{\alpha S}{|\Gamma_{0}|}\left(\frac{1}{|\Gamma_{0}|c_0^2}-\frac{|\Gamma_{0}|}{S^2}\right)\left(\sum_{i,j}(-1)^{j+1}\langle\psi_{0},\phi\rangle_{L^2(\Gamma_{0}^{(i,j)})}\right).
			\end{align*}
		\end{enumerate}
	\end{enumerate}
\end{lem}
\begin{proof}
	Throughout we apply Lemmas \ref{hd1v2} and \ref{2nddif}.\\  
	\begin{enumerate}[label=(\Roman*)]
		\item                                             \begin{enumerate}[label=(\roman*)]
			\item We have
			\begin{align*}
				\hat{h}_{\alpha}^{d_{a_j},d_{a_j}}[\psi_{0}]=&2\frac{c_0^2}{S^2}||\partial_1\psi_{0}||^2_{L^2(\Omega_0\cap\lbrace (-1)^jy<0\rbrace)}\\
				&+\frac{\alpha}{2}\left(\frac{1}{|\Gamma_{0}|^2}-\frac{c^2}{|\Gamma_{0}|^4}\right)||\psi_{0}||^2_{\partial\Omega_0}\\
				=&\frac{c_0^2}{2S^2}||\nabla\psi_{0}||^2_{L^2(\Omega_0)}+\frac{\alpha}{2}\left(\frac{1}{2S}-\frac{S}{4S^2}\right)||\psi_{0}||^2_{L^2(\partial\Omega_0)},\\
				=&\frac{1}{2S}||\nabla\psi_{0}||^2_{L^2(\Omega_0)}+\alpha\frac{1}{8S}||\psi_{0}||^2_{L^2(\partial\Omega_0)},
			\end{align*}
			where we have made use of the symmetry properties $||\partial_1\psi_{0}||^2_{L^2(\Omega_0\cap\lbrace (-1)^jy<0\rbrace)}=\frac{1}{2}||\partial_1\psi_{0}||^2_{L^2(\Omega_0)}=\frac{1}{4}||\nabla\psi_{0}||^2_{L^2(\Omega_0)},$ and $||\psi_0||_{L^2(\Gamma_{0}^{i,j})}$ is independent of $i,j$ from Corollary \ref{symprop}.\\
			
			\item       We have
			\begin{align*}
				\hat{h}_{\alpha}^{d_{c},d_{c}}[\psi_{0}]=&\frac{6}{c_0^2}||\partial_1\psi_{0}||^2_{L^2(\Omega_0)}+\frac{2}{c_0^2}||\partial_2\psi_{0}||^2_{L^2(\Omega_{0})}\\
				+&\frac{\alpha}{|\Gamma_{0}|}\left(\frac{3\frac{S^2}{c_0^4}+1}{|\Gamma_{0}|}-\frac{\left(-\frac{S^2}{c_0^3}+c_0\right)^2}{|\Gamma_{0}|^3}\right)||\psi_{0}||^2_{L^2(\partial\Omega_{0})}\\
				=&\frac{4}{S}||\nabla\psi_{0}||^2_{L^2(\Omega_0)}+\alpha\frac{2}{S}||\psi_{0}||^2_{L^2(\partial\Omega_{0})},
			\end{align*}
			where we have made use of the symmetry property $||\partial_1\psi_{0}||^2_{L^2(\Omega_{0})}=||\partial_2\psi_{0}||^2_{L^2(\Omega_{0})}=\frac{1}{2}||\nabla\psi_{0}||^2_{L^2(\Omega_{0})}$  from Corollary \ref{symprop}.\\                               \item      We have
			\begin{align*}
				\hat{h}_{\alpha}^{d_{S_1},d_{S_1}}[\psi_{0}]=&\frac{6}{S^2}||\partial_2\psi_{0}||^2_{L^2(\Omega_{0})}\\
				+&\frac{\alpha S}{|\Gamma_{0}|}\left(-\frac{S}{|\Gamma_{0}|^3c_0^4}-\frac{1}{c_0^2|\Gamma_{0}|S}+\frac{2|\Gamma_{0}|}{S^3}\right)||\psi_{0}||^2_{L^2(\partial\Omega_{0})}\\
				=&       \frac{3}{S^2}||\nabla\psi_{0}||^2_{L^2(\Omega_0)}+\frac{5\alpha}{4S^2}||\psi_{0}||^2_{L^2(\partial\Omega_{0})}                                               \end{align*}
			where we have made use of the symmetry properties $||\partial_1\psi_{0}||^2_{L^2(\Omega_0\cap\lbrace y>0\rbrace)}=||\partial_1\psi_{0}||^2_{L^2(\Omega_0\cap\lbrace y<0\rbrace)}=\frac{1}{2}||\partial_1\psi_{0}||^2_{L^2(\Omega_0)}=\frac{1}{4}||\nabla\psi_{0}||^2_{L^2(\Omega_0)},$ from Corollary \ref{symprop}.\\                            
			\item   We have
			\begin{align*}
				\hat{h}_{\alpha,a_1,a_2,c}^{d_{a_j},d_{c}}[\psi_{0}]=&(-1)^j\frac{2}{S}\langle\partial_1\psi_{0},\partial_2\psi_{0}\rangle_{L^2(\Omega_0\cap\lbrace (-1)^jy<0\rbrace)}\\
				+&\frac{\alpha}{|\Gamma_0|}\sum_{i=1}^2\left((-1)^{i+1}\frac{1}{|\Gamma_{0}|}-\frac{(-1)^{i+1}c_0(-\frac{S^2}{c_0^3}+c_0)}{|\Gamma_{0}|^3}\right)||\psi_{0}||_{\Gamma_{0}^{(i,j)}}^2\\
				=&0,
			\end{align*}
			where we have made use of the fact that $S=c_0^2$ and so $c_0-\frac{S^2}{c_0^3}=0,$ alongside the symmetry properties $\langle\partial_1\psi_{0},\partial_2 \psi_0\rangle_{L^2(\Omega_{0}\cap\lbrace (-1)^jy<0\rbrace)}=0$ and $||\psi_0||_{L^2(\Gamma_{0}^{(1,1)})}=||\psi_0||_{L^2(\Gamma_{0}^{(2,1)})}$ from Corollary \ref{symprop}.\\
			\item Trivial.\\                                                                                                                               
			
		\item We have
		\begin{align*}
			\hat{h}_{\alpha}^{d_{a_j},d_{S_1}}[\psi_{0}]=&4\frac{c_0}{S^2}\langle\partial_1\psi_{0},\partial_2\psi_{0}\rangle_{L^2(\Omega_{0}\cap\lbrace (-1)^jy<0\rbrace)}\\
			+&(-1)^j\frac{\alpha S}{|\Gamma_{0}^{(i)}|}c_0\left(\frac{1}{S^2|\Gamma_{0}|}+\frac{1}{|\Gamma_{0}|^3c_0^2}\right)\sum_{i=1}^2(-1)^{i+1}||\psi_{0}||^2_{L^2(\Gamma_{0}^{(i,j)})}\\
			=&0,
		\end{align*}
		where we have used the symmetry properties $\langle\partial_1\psi_{0},\partial_2 \psi_0\rangle_{L^2(\Omega_{0}\cap\lbrace (-1)^jy<0\rbrace)}=0$ and $||\psi_0||_{L^2(\Gamma_{0}^{(1,j)})}=||\psi_0||_{L^2(\Gamma_{0}^{(2,j)})}$ from Corollary \ref{symprop}.\\                                                                            \item We have
		\begin{align*}
			\hat{h}^{d_c,{d_{S_1}}}_{\alpha}[\psi_{0}]=&\frac{4}{c_0S}\left(||\partial_2\psi_{0}||^2_{L^2(\Omega_{0}^-)}-||\partial_2\psi_{0}||^2_{L^2(\Omega_{0}^+)}\right)\\
			+&\frac{\alpha S}{|\Gamma_{0}|}\left(\frac{\frac{S^2}{c_0^3}+c_0}{S^2|\Gamma^{(i)}_{0}|}+\frac{-\frac{S^2}{c_0^3}+c_0}{|\Gamma_{0}^{(i)}|^3c_0^2}\right)\left(\sum_{i,j}(-1)^j||\psi_{0}||^2_{L^2(\Gamma_{0}^{(i,j)})}\right)\\
			=&0.
		\end{align*}                       
		Where we have used the symmetry properties $\langle\partial_1\psi_{0},\partial_2 \psi_0\rangle_{L^2(\Omega_{0}^+)}=\langle\partial_1\psi_{0},\partial_2 \psi_0\rangle_{L^2(\Omega_{0}^-)}=0,$  $||\psi_0||_{L^2(\Gamma_{0}^{(1,j)})}=||\psi_0||_{L^2(\Gamma_{0}^{(2,j)})}$ from Corollary \ref{symprop}.\\
	\end{enumerate}                                                   \item Follows easily from Lemma \ref{hd1v2}.
\end{enumerate} 
\end{proof}                                                     
We now have all the computations required to compute the  $\lambda_{0}^{v_1,v_2}$ terms.     
\bibliography{refs2}
\bibliographystyle{plain}
\end{document}